\documentclass[english]{article}
\pdfoutput=1 
\RequirePackage[a-1b]{pdfx}
\RequirePackage[latin1,utf8]{inputenc}
\usepackage[LGR,T1]{fontenc}
\usepackage[latin1, utf8]{inputenc}
\usepackage{geometry}
\usepackage{color}
\usepackage{babel}
\usepackage{amsmath}
\usepackage{amsthm}
\usepackage{amssymb}
\usepackage{bm}
\usepackage{cancel}
\usepackage{stackrel}
\usepackage{graphicx}
\usepackage{epstopdf}
\usepackage{setspace}
\usepackage{esint}
\usepackage{xcolor}
\usepackage{amsfonts}
\usepackage{soul}
\usepackage{epstopdf}%

\makeatletter

\DeclareRobustCommand{\greektext}{%
  \fontencoding{LGR}\selectfont\def\encodingdefault{LGR}}
\DeclareRobustCommand{\textgreek}[1]{\leavevmode{\greektext #1}}
\ProvideTextCommand{\~}{LGR}[1]{\char126#1}

\numberwithin{equation}{section}
\numberwithin{figure}{section}
\newcommand{\lyxaddress}[1]{
	\par {\raggedright #1
	\vspace{1.4em}
	\noindent\par}
}
\theoremstyle{plain}
\newtheorem{thm}{\protect\theoremname}[section]
\theoremstyle{definition}
\newtheorem{defn}[thm]{\protect\definitionname}
\theoremstyle{remark}
\newtheorem{notation}[thm]{\protect\notationname}
\theoremstyle{definition}
\newtheorem{example}[thm]{\protect\examplename}
\theoremstyle{remark}
\newtheorem{rem}[thm]{\protect\remarkname}
\theoremstyle{plain}
\newtheorem{prop}[thm]{\protect\propositionname}
\theoremstyle{plain}
\newtheorem{question}[thm]{\protect\questionname}
\theoremstyle{plain}
\newtheorem{lem}[thm]{\protect\lemmaname}
\newenvironment{lyxlist}[1]
	{\begin{list}{}
		{\settowidth{\labelwidth}{#1}
		 \setlength{\leftmargin}{\labelwidth}
		 \addtolength{\leftmargin}{\labelsep}
		 }}
	{\end{list}}
\theoremstyle{plain}
\newtheorem{cor}[thm]{\protect\corollaryname}
\theoremstyle{plain}
\newtheorem{conjecture}[thm]{\protect\conjecturename}

\@ifundefined{showcaptionsetup}{}{%
 \PassOptionsToPackage{caption=false}{subfig}}
\usepackage{subfig}
\makeatother

\providecommand{\conjecturename}{Conjecture}
\providecommand{\corollaryname}{Corollary}
\providecommand{\definitionname}{Definition}
\providecommand{\examplename}{Example}
\providecommand{\lemmaname}{Lemma}
\providecommand{\notationname}{Notation}
\providecommand{\propositionname}{Proposition}
\providecommand{\questionname}{Question}
\providecommand{\remarkname}{Remark}
\providecommand{\theoremname}{Theorem}

\begin{document}

\title{New constructions of non-regular cospectral graphs }
\author{Suliman Hamud and Abraham Berman }
\maketitle

\lyxaddress{\begin{center}
sulimanhamud@campus.technion.ac.il, berman@technion.ac.il
\par\end{center}}

\begin{center}
Deticated to Prof. Frank J. Hall, in celebration of his many contributions
to matrix theory.
\par\end{center}

\section{Abstract}

We consider two types of joins of graphs $G_{1}$ and $G_{2}$, $G_{1}\veebar G_{2}$
- the Neighbors Splitting Join and $G_{1}\underset{=}{\lor}G_{2}$
- the Non Neighbors Splitting Join, and compute the adjacency characteristic
polynomial, the Laplacian characteristic polynomial and the signless
Laplacian characteristic polynomial of these joins. When $G_{1}$
and $G_{2}$ are regular, we compute the adjacency spectrum, the Laplacian
spectrum, the signless Laplacian spectrum of $G_{1}\underset{=}{\lor}G_{2}$
and the normalized Laplacian spectrum of $G_{1}\veebar G_{2}$ and
$G_{1}\underset{=}{\lor}G_{2}$. We use these results to construct
non regular, non isomorphic graphs that are cospectral with respect
to the four matrices: adjacency, Laplacian , signless Laplacian and
normalized Laplacian.

.

\section{Introduction}

Spectral graph theory is the study of graphs via the spectrum of matrices
associated with them \textbf{\cite{brouwer2011spectra,chung1997spectral,cvetkovivc2010introduction,nica2016brief,spielman2007spectral}}.
The graphs in this paper are undirected and simple. There are several
matrices associated with a graph and we consider four of them; the
adjacency matrix, the Laplacian matrix, the signless Laplacian matrix
and the normalized Laplacian matrix. 

Let $G=\left(V\left(G\right),E\left(G\right)\right)$ be a graph with
vertex set $V\left(G\right)=\left\{ v_{1},v_{2},...,v_{n}\right\} $
and edge set $E\left(G\right)$. 
\begin{defn}
The adjacency matrix of $G$, $A\left(G\right),\textrm{is}$ defined
by;

\[
\left(A(G)\right)_{ij}=\begin{cases}
1, & \textrm{if \ensuremath{v_{i}} and \ensuremath{v_{j}} are adjacent};\\
0, & \textrm{otherwise.}
\end{cases}
\]
\end{defn}

Let $d_{i}=d_{G}\left(v_{i}\right)$ be the degree of vertex $v_{i}$
in $G$, and let $D\left(G\right)$ be the diagonal matrix with diagonal
entries $d_{1},d_{2},...,d_{n}.$ 
\begin{defn}
The Laplacian matrix, $L\left(G\right),$ and the signless Laplacian
matrix, $Q\left(G\right),$ of $G$ are defined as $L\left(G\right)=D\left(G\right)-A\left(G\right)$
and $Q\left(G\right)=D\left(G\right)+A\left(G\right)$.
\end{defn}

\begin{defn}
(\textbf{\cite{chung1997spectral}}) The normalized Laplacian matrix,
$\mathcal{L}\left(G\right),$is defined to be $\mathcal{L}\left(G\right)=I_{n}-D\left(G\right)^{-\frac{1}{2}}A\left(G\right)D\left(G\right)^{-\frac{1}{2}}$
(with the convention that if the degree of vertex $v_{i}$ in $G$
is $0$, then $(d_{i})^{-\frac{1}{2}}=0$). In other words, 
\[
\left(\mathcal{L}(G)\right)_{ij}=\begin{cases}
1, & if\text{ i=j and \ensuremath{d_{i}\neq0;}}\\
-\frac{1}{\sqrt{d_{i}d_{j}}}, & \textrm{if \ensuremath{i\neq j} and \ensuremath{v_{i}} is adjacent to \ensuremath{v_{j};} }\\
0, & \textrm{otherwise}.
\end{cases}
\]
 
\end{defn}

\begin{notation}
For an $n\times n$ matrix $M$, we denote the characteristic polynomial
$det\left(xI_{n}-M\right)$ of $M$ by $f_{M}\left(x\right)$, where
$I_{n}$ is the identity matrix of order $n$. In particular, for
a graph $G$, $f_{X\left(G\right)}\left(x\right)$ is the $X\textrm{-characteristic}$
polynomial of $G$, for $X\in\left\{ A,L,Q,\mathcal{L}\right\} $.
The roots of the $X\textrm{-characteristic}$ polynomial of $G$ are
the $X\textrm{-eigenvalues}$ of $G$ and the collection of the $X\textrm{-eigenvalues,}$
including multiplicities, is called the $X\textrm{-spectrum of \ensuremath{G.}}$ 
\end{notation}

\begin{notation}
The multiplicity of an eigenvalue $\lambda$ is denoted by a superscript
above $\lambda$.
\end{notation}

\begin{example}
The $A$-spectrum of the complete graph $K_{n}$ is $\left\{ n-1,(-1)^{[n-1]}\right\} $.
\end{example}

\begin{rem}
If
\[
\lambda_{1}\left(G\right)\geq\lambda_{2}\left(G\right)\geq\cdot\cdot\cdot\geq\lambda_{n}\left(G\right),
\]
\[
\mu_{1}\left(G\right)\leq\mu_{2}\left(G\right)\leq\cdot\cdot\cdot\leq\mu_{n}\left(G\right),
\]
\[
\nu_{1}\left(G\right)\geq\nu_{2}\left(G\right)\geq\cdot\cdot\cdot\geq\nu_{n}\left(G\right),
\]
\[
\delta_{1}\left(G\right)\leq\delta_{2}\left(G\right)\leq\cdot\cdot\cdot\leq\delta_{n}\left(G\right),
\]
are the eigenvalues of $A\left(G\right)$, $L\left(G\right)$, $Q\left(G\right)$
and $\mathcal{L}\left(G\right)$, respectively. Then $\stackrel[i=1]{n}{\sum\lambda_{i}}=0$,
$\mu_{1}\left(G\right)=0,$ $\nu_{n}\left(G\right)\geq0$ and $\delta_{1}\left(G\right)=0,$
$\delta_{n}\left(G\right)\leq2$ (equality iff $G$ is bipartite).
\end{rem}

\begin{rem}
\label{rem:ALQl}If $G$ is a r-regular graph, then $\mu_{i}\left(G\right)=r-\lambda_{i}\left(G\right)$,
$\nu_{i}\left(G\right)=r+\lambda_{i}\left(G\right)$ and $\delta_{i}\left(G\right)=1-\frac{1}{r}\lambda\left(G\right)$.
\end{rem}

\begin{defn}
Two graphs $G$ and $H$ are $X\textrm{-cospectral }$if they have
the same $X\textrm{-spectrum}$. If $X$-cospectral graphs are not
isomorphic we say that they are $\textrm{XNICS.}$
\end{defn}

\begin{defn}
Let $S$ be a subset of $\left\{ A,L,Q,\mathcal{L}\right\} .$ The
graphs $G$ and $H$ are $\textrm{SNICS}$ if they are $\textrm{XNICS}$
for all $X\in S.$ 
\end{defn}

\begin{defn}
A graph $G$ is determined by its $X\textrm{-spectrum}$ if every
graph $H$ that is $X\textrm{-cospectral}$ with $G$ is isomorphic
to $G.$

A basic probem in spectral graph theory, \textbf{\cite{van2003graphs,van2009developments},}
is determining which graphs are determined by their spectrum or finding
non isomorphic $X$-cospectral graphs. 
\end{defn}

\begin{thm}
\textup{(}\textbf{\textup{\cite{van2003graphs}}}\textup{)}\emph{
}\textup{\emph{If $G$ is regular , then the following are equivalent;}}
\end{thm}

\begin{itemize}
\item \emph{$G$ is determined by its $A$-spectrum,}
\item \emph{$G$ is determined by its $L$-spectrum,}
\item \emph{$G$ is determined by its $Q$-spectrum,}
\item \emph{$G$ is determined by its $\mathcal{L}$-spectrum.}
\end{itemize}
Thus, for regular graphs $G$ and $H$, we say that $G$ and $H$
are cospectral if they are $X$-cospectral with respect to any $X\in\left\{ A,L,Q,\mathcal{L}\right\} .$ 
\begin{prop}
\textup{(}\textbf{\textup{\cite{van2003graphs}}}\textup{) }\textup{\emph{Every
regular graph with less than ten vertices is determined by its spectrum. }}

\newpage{}
\end{prop}

\begin{example}
The following graphs are regular and cospectral. They are non isomorphic
since in $G$ there is an edge that lies in three triangles but there
is no such edge in $H$.

\begin{figure}[h!]
\begin{centering}
\subfloat[$\left(G\right)$]{\includegraphics[width=0.36\textwidth]{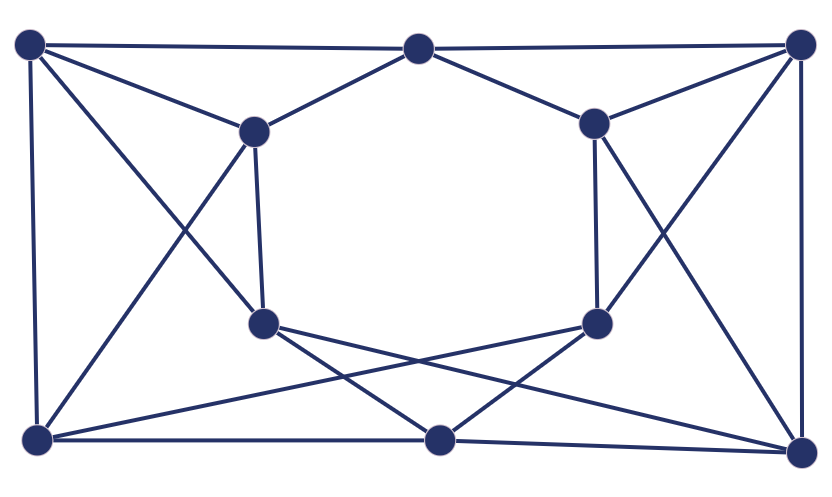}}\subfloat[$\left(H\right)$]{\includegraphics[width=0.36\textwidth]{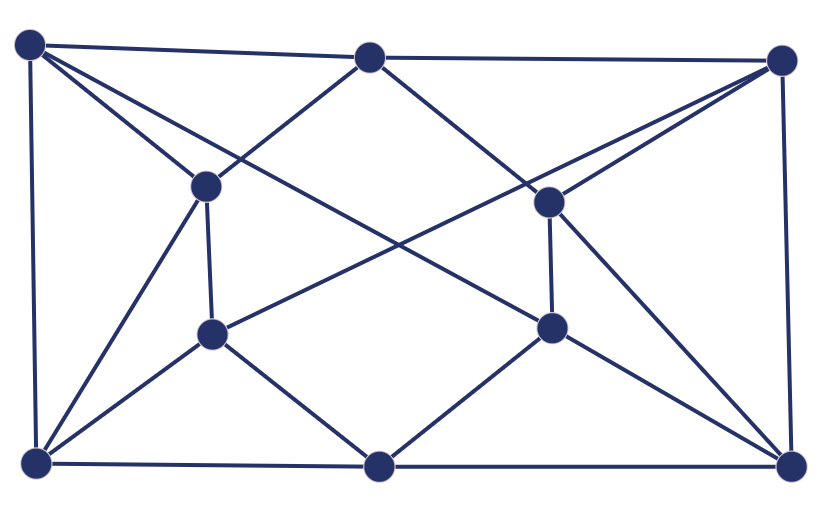}}
\par\end{centering}
\caption{\label{fig:Two-regular-non isomorphic cospectral  10 vertices}Two
regular non isomorphic cospectral graphs.}
\end{figure}
\end{example}

In recent years, several researchers studied the spectral properties
of graphs which are constructed by graph operations. These operations
include disjoint union, the Cartesian product, the Kronocker product,
the strong product, the lexicographic product, the rooted product,
the corona, the edge corona, the neighbourhood corona etc. We refer
the reader to\textbf{ \cite{barik2015laplacian,barik2007spectrum,cvetkovivc2010introduction,gopalapillai2011spectrum,hou2010spectrum,mohar1992laplace,cvetkovic1975spectra,godsil1978new,pavithra2022spectra,pavithra2021spectra,rajkumar2019spectra,rajkumar2022spectra}}
and the references therein for more graph operations and the results
on the spectra of these graphs. 

Many operations are based on the join of graphs.
\begin{defn}
(\textbf{\cite{harary1969graph}}) The join of two graphs is their
disjoint union together with all the edges that connect all the vertices
of the first graph with all the vertices of the second graph.
\end{defn}

Recently, many researchers provided several variants of join operations
of graphs and investigated their spectral properties. Some examples
are Cardoso\textbf{ \cite{cardoso2013spectra},} Indulal\textbf{ \cite{indulal2012spectrum}},
Liu and Zhang \textbf{\cite{liu2019spectra}}, Varghese and Susha
\textbf{\cite{varghese2020normalized}} and Das and Panigrahi \textbf{\cite{das2019new}}.

Butler \textbf{\cite{butler2010note}} constructed non regular bipartite
graphs which are cospectral with respect to both the adjacency and
the normalized Laplacian matrices. He asked for examples of non-regular
$\left\{ A,L,\mathcal{L}\right\} $NICS graphs. A slightly more general
question is 
\begin{question}
\label{que:Construct-non-regular}Construct non regular $\left\{ A,L,Q,\mathcal{L}\right\} $NICS
graphs. 
\end{question}

Such examples can be constructed using special join operation defined
by Lu, Ma and Zhang \textbf{\cite{lu2023spectra}} and a variant of
this operation, suggested in this paper.
\begin{defn}
(\textbf{\cite{lu2023spectra}}) Let $G_{1}$ and $G_{2}$ be two
vertex disjoint graphs with $V(G_{1})=\{u_{1},u_{2},...,u_{n}\}.$The
splitting $V$-vertex join of $G_{1}$ and $G_{2}$, denoted by $G_{1}\veebar G$,
is obtained by adding vertices $u'_{1},u'_{2},...,u'_{n}$ to $G_{1}\vee G_{2}$
and connecting $u_{i}'$ to $u_{j}$ if and only if $\left(u_{i},u_{j}\right)\in E\left(G_{1}\right)$. 
\end{defn}

We refer to the splitting $V$-vertex join as NS (Neighbors Splitting)
join and define a new type of join, NNS (Non Neighbors Splitting)
join.
\begin{defn}
\label{def:NNS}Let $G_{1}$ and $G_{2}$ be two vertex disjoint graphs
with $V\left(G_{1}\right)=\left\{ u_{1},u_{2},...,u_{n}\right\} .$
The NNS join of $G_{1}$ and $G_{2}$, denoted by $G_{1}\underset{=}{\vee}G_{2}$,
is obtained by adding vertices $u'_{1},u'_{2},...,u'_{n}$ to $G\vee G_{2}$
and connecting $u_{i}'$ to $u_{j}$ iff $\left(u_{i},u_{j}\right)\cancel{\in}E\left(G_{1}\right).$ 

\newpage{}
\end{defn}

\begin{example}
Let G$_{1}$ and $G_{2}$ be the path $P_{4}$ and the path $P_{2}$,
respectively. The graphs $P_{4}\underset{=}{\vee}P_{2}$ and $P_{4}\veebar P_{2}$
are given in Figure\textbf{ \ref{fig:Two-graphs join -}}.

\begin{figure}[h]
\begin{centering}
\subfloat[$P_{4}\veebar P_{2}$]{\includegraphics[width=0.3\textwidth]{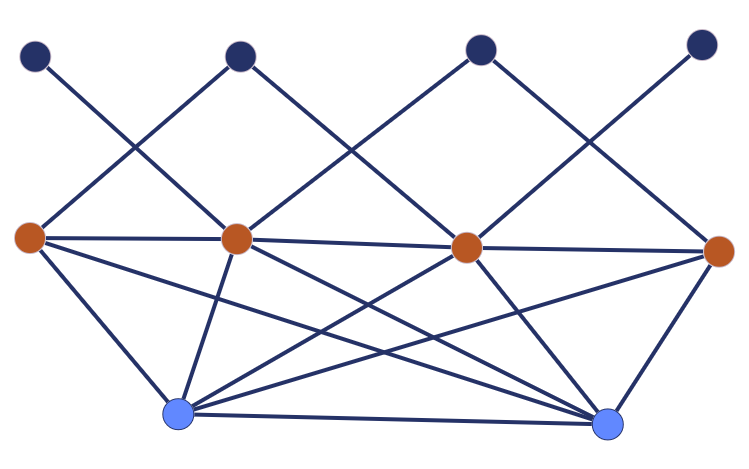}

\hspace{7bp}}\subfloat[$P_{4}\protect\underset{=}{\vee}P_{2}$]{\includegraphics[width=0.32\textwidth]{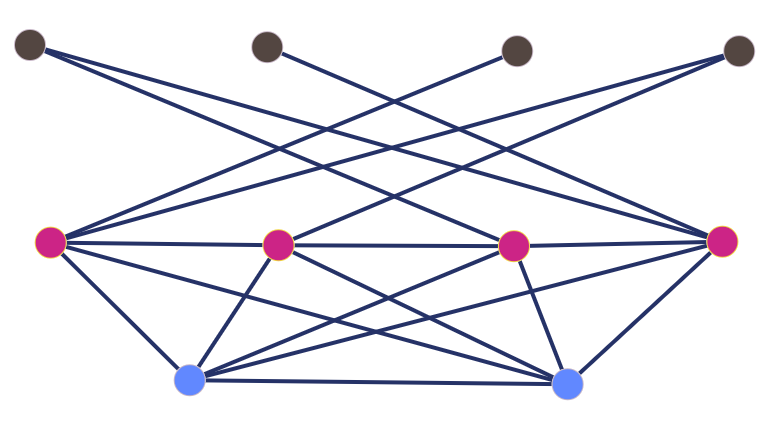}

}
\par\end{centering}
\caption{\label{fig:Two-graphs join -}The graphs $P_{4}\veebar P_{2}$ and
$P_{4}\protect\underset{=}{\vee}P_{2}$. }
\end{figure}
\end{example}

The structure of the paper is as follows; After preliminaries, we
compute the adjacency characteristic polynomial, the Laplacian characteristic
polynomial and the signless Laplacian characteristic polynomial of
$G_{1}\underset{=}{\vee}G_{2}$ and $G_{1}\veebar G_{2}$, and use
the results to construct $\left\{ A,L,Q\right\} $NICS graphs, and
finally, under regularity assumptions we compute the $A$-spectrum,
the $L$-spectrum, the $Q$-spectrum and the $\mathcal{L}$-spectrum
of NS and NNS joins and use the results to construct $\left\{ A,L,Q,\mathcal{L}\right\} $NICS
graphs.

\section{Preliminaries }
\begin{notation}
\end{notation}

\begin{itemize}
\item \emph{$\mathbf{1}_{n}$ denotes $n$$\times1$ column whose all entries
are $1$, }
\item \emph{$J_{s\times t}$$=\mathbf{1}_{s}\mathbf{1}_{t}^{T}$, $J_{s}=J_{s\times s},$}
\item \emph{$O_{s\times t}$ denotes the zero matrix of order $s\times t$,}
\item \emph{$adj\left(A\right)$ denotes the adjugate of $A.$}
\item $\overline{G}$ \emph{denotes the complement of graph $G$.}
\end{itemize}
\begin{defn}
\textbf{\label{def:The-M-coronal-=000393(x)}}(\textbf{\cite{cui2012spectrum,mcleman2011spectra}})
The coronal $\Gamma_{M}(x)$ of an n$\times$n matrix $M$ is the
sum of the entries of the inverse of the characteristic matrix of
$M$ , that is, 
\begin{equation}
\Gamma_{M}(x)=\mathbf{1}_{n}^{T}(xI_{n}-M)^{-1}\mathbf{1}_{n}.\label{eq:2.1}
\end{equation}
\end{defn}

\begin{lem}
\textbf{\textup{\label{lem:()-If-M coronal}(\cite{cui2012spectrum,mcleman2011spectra})}}
Let $M$ be an n$\times$n matrix with all row sums equal to r ( for
example, the adjacency matrix of a $r$-regular graph). Then
\end{lem}

\textbf{
\[
\Gamma_{M}(x)=\frac{n}{x-r}.
\]
}
\begin{defn}
Let $M$ be a block matrix 
\[
M=\left(\begin{array}{cc}
A & \,\,\,\,\,\,\,\,\,\,\,\,\,\,\,B\\
\\
C & \,\,\,\,\,\,\,\,\,\,\,\,\,\,D
\end{array}\right)
\]
 such that its blocks $A$ and $D$ are square. If $A$ is invertible,
the Schur complement of $A$ in $M$ is 
\[
M/A=D-CA^{-1}B
\]
and if $D$ is invertible, the Schur complement of $D$ in $M$ is
\[
M/D=A-BD^{-1}C.
\]
\end{defn}

Issai Schur proved the following lemma. 
\begin{lem}
\emph{\label{lem:(Schur-Complements,-).2.1}(}\textbf{\emph{\cite{horn2005basic})}}\emph{.
}\textup{\emph{If $D$ is invertible then, }}

\textup{\emph{
\begin{align*}
detM & =det(M/D)detD
\end{align*}
}}

\textup{\emph{and if $A$ is invertible then, 
\[
detM=det(M/A)detA.
\]
}}
\end{lem}

\begin{lem}
\textup{\label{lem:result of SHUR}}\textup{\emph{Let}}\emph{ }\textup{\emph{$M$
be a block matrix}}\emph{
\[
M=\left(\begin{array}{ccc}
A & \,\,\,\,\,\,\,\,\,\,\,\,\,\,B & \,\,\,\,\,\,\,\,\,\,\,\,\,\,\,\,\,\,J_{n_{1}\times n_{2}}\\
\\
B & \,\,\,\,\,\,\,\,\,\,\,\,\,\,C & \,\,\,\,\,\,\,\,\,\,\,\,\,\,\,\,O_{n_{1}\times n_{2}}\\
\\
J_{n_{2}\times n_{1}} & \,\,\,\,\,\,\,\,\,\,\,\,\,\,\,\,O_{n_{2}\times n_{1}} & \,\,\,\,\,\,\,\,\,\,\,\,\,\,\,D
\end{array}\right)
\]
}\textup{\emph{where $A$, $B$, and $C$ are square matrices of order
$n_{1}$ and $D$ is a square matrix of order $n_{2}.$ Then the Schur
complement of $xI_{n_{2}}-D$ in the }}charactristic matrix of $M$
is

\emph{
\[
\left(\begin{array}{cc}
xI_{n_{1}}-A-\Gamma_{D}(x)J_{n_{1}}\,\,\,\,\,\,\,\,\,\, & -B\\
\\
-B & xI_{n_{1}}-C
\end{array}\right).
\]
}
\end{lem}

\begin{proof}
The charactristic matrix of $M$ is 
\[
xI_{2n_{1}+n_{2}}-M=\left(\begin{array}{ccc}
xI_{n_{1}}-A\,\,\,\,\,\,\,\, & -B\,\,\,\,\,\,\,\,\,\,\, & -J_{n_{1}\times n_{2}}\\
\\
-B & \,\,\,\,\,xI_{n_{1}}-C\,\,\,\,\,\,\,\,\,\, & O_{n_{1}\times n_{2}}\\
\\
-J_{n_{2}\times n_{1}} & O_{n_{2}\times n_{1}} & xI_{n_{2}}-D
\end{array}\right).
\]
 The Schur complement of $\left(xI_{n_{2}}-D\right)$ is
\begin{align*}
\left(xI_{2n_{1}+n_{2}}-M\right)\left/\left(xI_{n_{2}}-D\right)\right. & =\left(\begin{array}{cc}
xI_{n_{1}}-A\,\,\,\,\, & \,-B\\
\\
-B\,\,\,\, & xI_{n_{1}}-C
\end{array}\right)-\left(\begin{array}{c}
-J_{n_{1}\times n_{2}}\\
\\
O_{n_{1}\times n_{2}}
\end{array}\right)\left(xI_{n_{2}}-D\right)^{-1}\left(\begin{array}{cc}
-J_{n_{2}\times n_{1}} & O_{n_{2}\times n_{1}}\end{array}\right)\\
\\
 & =\left(\begin{array}{cc}
xI_{n_{1}}-A\,\,\,\,\, & -B\\
\\
-B\,\,\,\,\, & xI_{n_{1}}-C
\end{array}\right)-\left(\begin{array}{c}
\mathbf{1}_{n_{1}}\mathbf{1}_{n_{2}}^{T}\\
\\
O_{n_{1}\times n_{1}}
\end{array}\right)\left((xI_{n_{2}}-D)^{-1}\right)\left(\begin{array}{cc}
\text{\ensuremath{\mathbf{1}_{n_{2}}}\ensuremath{\ensuremath{\mathbf{1}_{n_{1}}^{T}}}}\,\, & O_{n_{1}\times n_{1}}\end{array}\right)\\
\\
 & =\left(\begin{array}{cc}
xI_{n_{1}}-A-\Gamma_{D}(x)J_{n_{1}}\,\,\,\,\,\,\,\,\,\, & -B\\
\\
-B & xI_{n_{1}}-C
\end{array}\right).
\end{align*}
 
\end{proof}
\begin{lem}
\textbf{\emph{\label{lem:det(a+aJ)}(\cite{cvetkovivc2010introduction})}}.
If A is an n$\times n$ real matrix and \textgreek{a} is an real number,
then \textcolor{black}{
\begin{equation}
det(A+\alpha J_{n})=det(A)+\alpha\mathbf{1}_{n}^{T}adj(A)\mathbf{1}_{n}.\label{eq:det(A+alpha}
\end{equation}
}
\end{lem}

\section{The characteristic polynomials of the NNS join and the NS join }

In \textbf{\cite{lu2023spectra} }the authers compute the\textbf{
}adjacency, Laplacian and signless Laplacian characteristic polynomial
of $G_{1}\veebar G_{2}$ where $G_{1}$ and $G_{2}$ are regular.\\
Here we compute the characteristic polynomials of $G_{1}\underset{=}{\lor}G_{2}$
and $G_{1}\veebar G_{2}$ where $G_{1}$ and $G_{2}$ are arbitrary
graphs. The proofs for the two joins (NS and NNS) are quite similar
and use Lemma\textbf{ \ref{lem:(Schur-Complements,-).2.1}} (twice),
Lemma \textbf{\ref{lem:result of SHUR}} and Lemma \textbf{\ref{lem:det(a+aJ)}}.
The results are used to construct non regular $\left\{ A,L,Q\right\} $NICS
graphs. 

\subsection{Adjacency characteristic polynomial}
\begin{thm}
\textup{\label{thm:4.1 Adjacency characteristic polynomial}}\textup{\emph{Let
$G_{i}$ be a graph on $n_{i}$ vertices for $i=1,2$. Then}}
\end{thm}

\begin{lyxlist}{00.00.0000}
\item [{\emph{(a)}}] \emph{$f_{A(G_{1}\underset{=}{\vee}G_{2})}(x)=x^{n_{1}}f_{_{A\left(G_{2}\right)}}(x)det\left(xI_{n_{1}}-A(G_{1})-\frac{1}{x}A^{2}(\overline{G_{1}})\right)\left[1-\varGamma_{A(G_{2})}(x)\varGamma_{A(G_{1})+\frac{1}{x}A^{2}(\overline{G_{1}})}(x)\right].$}
\item [{\emph{(b)}}] \emph{$f_{A(G_{1}\veebar G_{2})}(x)=x^{n_{1}}f_{_{A\left(G_{2}\right)}}(x)det\left(xI_{n_{1}}-A(G_{1})-\frac{1}{x}A^{2}(G_{1})\right)\left[1-\varGamma_{A(G_{2})}(x)\varGamma_{A(G_{1})+\frac{1}{x}A^{2}(G_{1})}(x)\right]$}.
\end{lyxlist}
\begin{proof}
We prove (a). The proof of (b) is similar.

With a suitable ordering of the vertices of $G_{1}\underset{=}{\lor}G_{2}$,
we get 
\begin{align*}
\\
A\left(G_{1}\underset{=}{\vee}G_{2}\right)= & \left(\begin{array}{ccc}
A\left(G_{1}\right) & \,\,\,\,A\left(\overline{G_{1}}\right) & \,\,\,\,\,J_{n_{1}\times n_{2}}\\
\\
A\left(\stackrel{\text{\_\_\_}}{G_{1}}\right) & \,\,\,\,\,\,O_{n_{1}\times n_{1}} & \,\,\,\,\,\,O_{n_{1}\times n_{2}}\\
\\
J_{n_{2}\times n_{1}} & \,\,\,\,\,\,O_{n_{2}\times n_{1}} & \,\,\,\,A\left(G_{2}\right)
\end{array}\right).
\end{align*}
Thus 
\begin{align*}
f_{A(G_{1}\underset{=}{\vee}G_{2})}(x)= & det\left(xI_{2n_{1}+n_{2}}-A\left(G_{1}\underset{=}{\vee}G_{2}\right)\right)\\
\\
= & det\left(\begin{array}{cc|c}
xI_{n_{1}}-A(G_{1})\,\,\,\,\,\,\,\, & -A(\overline{G_{1}}) & -J_{n_{1}\times n_{2}}\\
 & \\
-A(\overline{G_{1}})\,\,\,\,\,\,\, & xI_{n_{1}} & O_{n_{1}\times n_{2}}\\
 & \\
\hline -J_{n_{2}\times n_{1}} & O_{n_{2}\times n_{1}} & xI_{n_{2}}-A(G_{2})
\end{array}\right)\\
\end{align*}
\[
\,\,\,\,\,\,\,\,\,\,\,\,\,\,\,\,\,\,\,\,\,\,\,\,\,\,\,\,\,\,\,\,\,\,\,\,\,\,\,\,\,\,\,\,\,\,\,\,\,\,\,\,\,\,\,\,\,\,\,\,\,\,\,\,\,\,\,\,\,\,\,\,\,\,\,\,\,\,\,\,\,\,\,\,\,\,\,\,\,\,\,\,\,=det\left(xI_{n_{2}}-A(G_{2})\right)det\left(\left(xI_{2n_{1}+n_{2}}-A\left(G_{1}\underset{=}{\vee}G_{2}\right)\right)\left/\left(xI_{n_{2}}-A(G_{2})\right)\right.\right),
\]

by the Lemma of Schur (Lemma\textbf{ \ref{lem:(Schur-Complements,-).2.1}})\textbf{.}\\

By Lemma\textbf{ \ref{lem:result of SHUR},}

\[
\left(xI_{2n_{1}+n_{2}}-A\left(G_{1}\underset{=}{\vee}G_{2}\right)\right)\left/\left(xI_{n_{2}}-A(G_{2})\right)\right.=\begin{pmatrix}xI_{n_{1}}-A(G_{1})-\varGamma_{A(G_{2})}(x)\,J_{n_{1}\times n_{1}}\,\,\,\,\,\,\,\,\,\,\, & -A(\overline{G_{1}})\\
\\
\\
-A(\overline{G_{1}}) & xI_{n_{1}}
\end{pmatrix}.
\]
 Using again Lemma \textbf{\ref{lem:(Schur-Complements,-).2.1}},
we get

\[
det\left(\left(xI_{2n_{1}+n_{2}}-A\left(G_{1}\underset{=}{\vee}G_{2}\right)\right)\left/\left(xI_{n_{2}}-A(G_{2})\right)\right.\right)=det\left(xI_{n_{1}}\right)det\left(xI_{n_{1}}-A(G_{1})-\varGamma_{A(G_{2})}(x)\,J_{n_{1}\times n_{1}}-\frac{1}{x}A^{2}(\overline{G_{1}})\right).
\]
\\
By Lemma \textbf{\ref{lem:det(a+aJ)}}, we get

\begin{align*}
det\left(\left(xI_{2n_{1}+n_{2}}-A\left(G_{1}\underset{=}{\vee}G_{2}\right)\right)\left/\left(xI_{n_{2}}-A(G_{2})\right)\right.\right)=\,\,\,\,\,\,\,\,\,\,\,\,\,\,\,\,\,\,\,\,\,\,\,\,\,\,\,\,\,\,\,\,\,\,\,\,\,\,\,\,\,\,\,\,\,\,\,\,\,\,\,\,\,\,\,\,\,\,\,\,\,\,\,\,\,\,\,\,\,\,\,\,\,\,\,\,\,\,\,\,\,\,\,\,\,\,\,\,\,\\
\\
=x^{n_{1}}\left(det(xI_{n_{1}}-A(G_{1})-\frac{1}{x}A^{2}(\overline{G_{1}}))-\varGamma_{A(G_{2})}(x)1_{n_{1}}^{T}adj\left(xI_{n_{1}}-A(G_{1})-\frac{1}{x}A^{2}(\overline{G_{1}})\right)1_{n_{1}}\right)\\
=x^{n_{1}}det\left(xI_{n_{1}}-A(G_{1})-\frac{1}{x}A^{2}(\overline{G_{1}})\right)\left[1-\varGamma_{A(G_{2})}(x)1_{n_{1}}^{T}(xI_{n_{1}}-A(G_{1})-\frac{1}{x}A^{2}(\overline{G_{1}}))^{-1}1_{n_{1}}\right]\\
\,\,\,=x^{n_{1}}det\left(xI_{n_{1}}-A(G_{1})-\frac{1}{x}A^{2}(\overline{G_{1}})\right)\left[1-\varGamma_{A(G_{2})}(x)\varGamma_{A(G_{1})+\frac{1}{x}A^{2}(\overline{G_{1}})}(x)\right].\,\,\,\,\,\,\,\,\,\,\,\,\,\,\,\,\,\,\,\,\,\,\,\,\,\,\,\,\,\,\,\,\,\,\,\,\,\,\,\,\,\,\,\,\,\,
\end{align*}
Thus
\[
f_{A(G_{1}\underset{=}{\vee}G_{2})}(x)=x^{n_{1}}f_{_{A\left(G_{2}\right)}}(x)det\left(xI_{n_{1}}-A(G_{1})-\frac{1}{x}A^{2}(\overline{G_{1}})\right)\left[1-\varGamma_{A(G_{2})}(x)\varGamma_{A(G_{1})+\frac{1}{x}A^{2}(\overline{G_{1}})}(x)\right].
\]
 
\end{proof}

\subsection{Laplacian characteristic polynomial}

In this section, we derive the Laplacian characteristic polynomials
of $G_{1}\veebar G_{2}$ and $G_{1}\underset{=}{\lor}G_{2}$ when
$G_{1}$and $G_{2}$ are arbitrary graphs. 
\begin{thm}
\textup{\label{thm: CHRAC. POLYNOMIAL Laplacian}}\textup{\emph{Let
$G_{i}$ be a graph on $n_{i}$ vertices for $i=1,2$. Then}}
\end{thm}

\begin{lyxlist}{00.00.0000}
\item [{\emph{(a)}}] \emph{
\begin{align*}
f_{_{L\left(G_{1}\underset{=}{\vee}G_{2}\right)}}\left(x\right) & =det\left(\left(x-n_{1}\right)I_{n_{2}}-L\left(G_{2}\right)\right)det\left(\left(x-n_{1}+1\right)I_{n_{1}}+D\left(G_{1}\right)\right)\\
 & \cdot det\left(\left(x-n_{1}-n_{2}+1\right)I_{n_{1}}-L\left(G_{1}\right)+D\left(G_{1}\right)-A\left(\overline{G_{1}}\right)\left(\left(x-n_{1}+1\right)I_{n_{1}}+D\left(G_{1}\right)\right)^{-1}A\left(\overline{G_{1}}\right)\right)\\
 & \cdot\left[1-\varGamma_{_{L\left(G_{2}\right)}}\left(x-n_{1}\right)\varGamma_{_{L\left(G_{1}\right)-D\left(G_{1}\right)+A\left(\overline{G_{1}}\right)\left(\left(x-n_{1}+1\right)I_{n_{1}}+D\left(G_{1}\right)\right)^{-1}A\left(\overline{G_{1}}\right)}}\left(x-n_{1}-n_{2}+1\right)\right].
\end{align*}
}
\item [{\emph{(b)}}] \emph{
\begin{align*}
f_{_{L\left(G_{1}\veebar G_{2}\right)}} & =det\left(\left(x-n_{1}\right)I_{n_{2}}-L\left(G_{2}\right)\right)det\left(xI_{n_{1}}-D\left(G_{1}\right)\right)\\
 & \cdot det\left(\left(x-n_{2}\right)I_{n_{1}}-L\left(G_{1}\right)-D\left(G_{1}\right)-A\left(G_{1}\right)\left(xI_{n_{1}}-D\left(G_{1}\right)\right)^{-1}A\left(G_{1}\right)\right)\\
 & \cdot\left[1-\varGamma_{_{L\left(G_{2}\right)}}\left(x-n_{1}\right)\varGamma_{_{L\left(G_{1}\right)+D\left(G_{1}\right)+A\left(G_{1}\right)\left(xI_{n_{1}}-D\left(G_{1}\right)\right)^{-1}A\left(G_{1}\right)}}\left(x-n_{2}\right)\right].
\end{align*}
}
\end{lyxlist}
\begin{proof}
(a) With a suitable ordering of the vertices of $G_{1}\underset{=}{\lor}G_{2}$,
we get 

\begin{align*}
L\ensuremath{\left(G_{1}\underset{=}{\vee}G_{2}\right)}= & \ensuremath{\left(\begin{array}{ccc}
\left(n_{1}+n_{2}-1\right)I_{n_{1}}-A\left(G_{1}\right)\,\,\:\,\,\,\,\,\, & -A\left(\overline{G_{1}}\right)\,\,\,\,\,\,\,\,\,\,\,\,\,\,\, & -J_{n_{1}\times\,n_{2}}\\
\\
-A\left(\overline{G_{1}}\right) & \left(n_{1}-1\right)I_{n_{1}}-D\left(G_{1}\right)\,\,\,\,\,\,\,\, & O_{n_{1}\times\,n_{2}}\\
\\
-J_{n_{2}\times\,n_{1}} & O_{n_{2}\times\,n_{1}} & D\left(G_{2}\right)+n_{1}I_{n_{2}}-A\left(G_{2}\right)
\end{array}\right)}\\
\\
= & \left(\begin{array}{ccc}
\left(n_{1}+n_{2}-1\right)I_{n_{1}}+L\left(G_{1}\right)-D\left(G_{1}\right)\,\,\,\,\,\,\,\,\, & -A\left(\overline{G_{1}}\right)\,\,\,\,\,\,\,\,\, & -J_{n_{1}\times\,n_{2}}\\
\\
-A\left(\overline{G_{1}}\right) & \left(n_{1}-1\right)I_{n_{1}}-D\left(G_{1}\right)\,\,\,\,\,\, & O_{n_{1}\times\,n_{2}}\\
\\
-J_{n_{2}\times\,n_{1}} & O_{n_{2}\times\,n_{1}} & n_{1}I_{n_{2}}+L\left(G_{2}\right)
\end{array}\right).
\end{align*}
 The Laplacian characteristic polynomial is

\begin{align*}
f_{_{L\left(G_{1}\underset{=}{\vee}G_{2}\right)}}\left(x\right)= & det\left(xI_{2n_{1}+n_{2}}-L\left(G_{1}\underset{=}{\vee}G_{2}\right)\right)\\
\\
= & det\left(\begin{array}{cc|c}
\left(x-n_{1}-n_{2}+1\right)I_{n_{1}}-L\left(G_{1}\right)+D\left(G_{1}\right)\,\,\,\,\,\,\, & A\left(\overline{G_{1}}\right)\,\,\,\,\,\,\,\,\,\, & \,\,\,\,\,\,\,\,\,\,J_{n_{1}\times\,n_{2}}\\
 & \\
\,\,A\left(\overline{G_{1}}\right) & \left(x-n_{1}+1\right)I_{n_{1}}+D\left(G_{1}\right) & \,\,\,\,\,\,\,\,\,\,O_{n_{1}\times\,n_{2}}\\
\hline  & \\
\text{\emph{J}}_{n_{2}\times\,n_{1}} & O_{n_{2}\times\,n_{1}} & \left(x-n_{1}\right)I_{n_{2}}-L\left(G_{2}\right)
\end{array}\right)\\
\\
= & det\left(\left(x-n_{1}\right)I_{n_{2}}-L\left(G_{2}\right)\right)det\left(\left(xI_{2n_{1}+n_{2}}-L\left(G_{1}\underset{=}{\vee}G_{2}\right)\right)\left/\left((x-n_{1})I_{n_{2}}-L(G_{2})\right)\right.\right),
\end{align*}
 by the Lemma of Schur (Lemma\textbf{ \ref{lem:(Schur-Complements,-).2.1}})\textbf{.}
\\

By Lemma\textbf{ \ref{lem:result of SHUR}},

\begin{align*}
\left(xI_{2n_{1}+n_{2}}-L\left(G_{1}\underset{=}{\vee}G_{2}\right)\right)\left/\left((x-n_{1})I_{n_{2}}-L(G_{2})\right)=\,\,\,\,\,\,\,\,\,\,\,\,\,\,\,\,\,\,\,\,\,\,\,\,\,\,\,\,\,\,\,\,\,\,\,\,\,\,\,\,\,\,\,\,\,\,\,\,\,\,\,\,\,\,\,\,\,\,\,\,\,\,\,\,\,\,\,\,\,\,\,\,\,\,\,\,\,\,\,\,\,\,\,\,\,\,\,\,\,\,\,\,\,\,\,\,\,\,\,\,\,\,\,\,\,\,\,\,\,\,\,\,\,\,\right.\\
\\
=\left(\begin{array}{cc}
\left(x-n_{1}-n_{2}+1\right)I_{n_{1}}-L\left(G_{1}\right)+D\left(G_{1}\right)-\Gamma_{L\left(G_{2}\right)}\left(x-n_{1}\right)J_{n_{1}\times\,n_{1}} & A\left(\overline{G_{1}}\right)\\
\\
A\left(\overline{G_{1}}\right) & \left(x-n_{1}+1\right)I_{n_{1}}+D\left(G_{1}\right)
\end{array}\right).
\end{align*}
Using again Lemma \textbf{\ref{lem:(Schur-Complements,-).2.1}} ,
we get
\[
det\left(\left(xI_{2n_{1}+n_{2}}-L\left(G_{1}\underset{=}{\vee}G_{2}\right)\right)\left/\left((x-n_{1})I_{n_{2}}-L(G_{2})\right)\right.\right)=det\left(\left(x-n_{1}+1\right)I_{n_{1}}+D\left(G_{1}\right)\right)det\left(B-\Gamma_{L\left(G_{2}\right)}\left(x-n_{1}\right)J_{n_{1}\times\,n_{1}}\right)
\]
where 
\[
B=\left(x-n_{1}-n_{2}+1\right)I_{n_{1}}-L\left(G_{1}\right)+D\left(G_{1}\right)-A\left(\overline{G_{1}}\right)\left(\left(x-n_{1}+1\right)I_{n_{1}}+D\left(G_{1}\right)\right)^{-1}A\left(\overline{G_{1}}\right).
\]
 By Lemma \textbf{\ref{lem:det(a+aJ)}} we get
\begin{align*}
det\left(\left(xI_{2n_{1}+n_{2}}-L\left(G_{1}\underset{=}{\vee}G_{2}\right)\right)\left/\left((x-n_{1})I_{n_{2}}-L(G_{2})\right)\right.\right)=\,\,\,\,\,\,\,\,\,\,\,\,\,\,\,\,\,\,\,\,\,\,\,\,\,\,\,\,\,\,\,\,\,\,\,\,\,\,\,\,\,\,\,\,\,\,\,\,\,\,\,\,\,\,\,\,\,\,\,\,\,\,\,\,\,\,\,\,\,\,\,\,\,\,\,\,\,\,\,\,\,\,\,\,\,\,\,\,\,\,\,\,\,\,\,\,\,\,\,\, & \,\\
=det\left(\left(x-n_{1}+1\right)I_{n_{1}}+D\left(G_{1}\right)\right)\left(det\left(B\right)-\varGamma_{L\left(G_{2}\right)}\left(x-n_{1}\right)\mathbf{1}_{n_{1}}^{T}adj\left(B\right)\mathbf{1}_{n_{1}}\right)\,\,\,\,\,\,\,\,\,\,\,\,\,\,\,\,\,\,\,\,\,\,\,\,\,\,\,\,\,\,\,\,\,\,\,\,\,\,\,\,\,\,\,\,\,\,\,\,\,\,\,\,\,\,\,\,\\
=det\left(\left(x-n_{1}+1\right)I_{n_{1}}+D\left(G_{1}\right)\right)det\left(B\right)\left[1-\varGamma_{L\left(G_{2}\right)}\left(x-n_{1}\right)\mathbf{1}_{n_{1}}^{T}B^{-1}\mathbf{1}_{n_{1}}\right]\,\,\,\,\,\,\,\,\,\,\,\,\,\,\,\,\,\,\,\,\,\,\,\,\,\,\,\,\,\,\,\,\,\,\,\,\,\,\,\,\,\,\,\,\,\,\,\,\,\,\,\,\,\,\,\,\,\,\,\,\,\\
\,\,=det\left(\left(x-n_{1}+1\right)I_{n_{1}}+D\left(G_{1}\right)\right)\,\,\,\,\,\,\,\,\,\,\,\,\,\,\,\,\,\,\,\,\,\,\,\,\,\,\,\,\,\,\,\,\,\,\,\,\,\,\,\,\,\,\,\,\,\,\,\,\,\,\,\,\,\,\,\,\,\,\,\,\,\,\,\,\,\,\,\,\,\,\,\,\,\,\,\,\,\,\,\,\,\,\,\,\,\,\,\,\,\,\,\,\,\,\,\,\,\,\,\,\,\,\,\,\,\,\,\,\,\,\,\,\,\,\,\,\,\,\,\,\,\,\,\,\,\,\,\,\,\,\,\,\,\,\,\,\,\,\,\,\,\,\,\,\,\,\,\,\,\,\,\,\,\,\,\,\,\,\,\,\,\,\,\\
\cdot det\left(\left(x-n_{1}-n_{2}+1\right)I_{n_{1}}-L\left(G_{1}\right)+D\left(G_{1}\right)-A\left(\overline{G_{1}}\right)\left(\left(x-n_{1}+1\right)I_{n_{1}}+D\left(G_{1}\right)\right)^{-1}A\left(\overline{G_{1}}\right)\right)\,\,\,\,\,\, & \,\,\,\\
\cdot\left[1-\varGamma_{_{L\left(G_{2}\right)}}\left(x-n_{1}\right)\varGamma_{_{L\left(G_{1}\right)-D\left(G_{1}\right)+A\left(\overline{G_{1}}\right)\left(\left(x-n_{1}+1\right)I_{n_{1}}+D\left(G_{1}\right)\right)^{-1}A\left(\overline{G_{1}}\right)}}\left(x-n_{1}-n_{2}+1\right)\right].\,\,\,\,\,\,\,\,\,\,\,\,\,\,\,\,\,\,\,\,\,\,\,\,\,\,\,\,\,\,\,\,\,\,\,\,\,\,\,\,\,\,\,\,\,\\
\,
\end{align*}
 Thus

\begin{align*}
f_{_{L\left(G_{1}\underset{=}{\vee}G_{2}\right)}}\left(x\right) & =det\left(\left(x-n_{1}\right)I_{n_{2}}-L\left(G_{2}\right)\right)det\left(\left(x-n_{1}+1\right)I_{n_{1}}+D\left(G_{1}\right)\right)\\
 & \cdot det\left(\left(x-n_{1}-n_{2}+1\right)I_{n_{1}}-L\left(G_{1}\right)+D\left(G_{1}\right)-A\left(\overline{G_{1}}\right)\left(\left(x-n_{1}+1\right)I_{n_{1}}+D\left(G_{1}\right)\right)^{-1}A\left(\overline{G_{1}}\right)\right)\\
 & \cdot\left[1-\varGamma_{_{L\left(G_{2}\right)}}\left(x-n_{1}\right)\varGamma_{_{L\left(G_{1}\right)-D\left(G_{1}\right)+A\left(\overline{G_{1}}\right)\left(\left(x-n_{1}+1\right)I_{n_{1}}+D\left(G_{1}\right)\right)^{-1}A\left(\overline{G_{1}}\right)}}\left(x-n_{1}-n_{2}+1\right)\right].
\end{align*}
The proof of (b) is similar.
\end{proof}

\subsection{Signless Laplacian characteristic polynomial }
\begin{thm}
\textup{\label{thm: signless laplacian characteristic}}\textup{\emph{Let
$G_{i}$ be a graph on $n_{i}$ vertices for $i=1,2$. Then}}
\end{thm}

\begin{lyxlist}{00.00.0000}
\item [{\emph{(a)}}] \emph{
\begin{align*}
f_{_{Q\left(G_{1}\underset{=}{\vee}G_{2}\right)}}\left(x\right) & =det\left(\left(x-n_{1}\right)I_{n_{2}}-Q\left(G_{2}\right)\right)det\left(\left(x-n_{1}+1\right)I_{n_{1}}+D\left(G_{1}\right)\right)\\
 & \cdot det\left(\left(x-n_{1}-n_{2}+1\right)I_{n_{1}}-Q\left(G_{1}\right)+D\left(G_{1}\right)-A\left(\overline{G_{1}}\right)\left(\left(x-n_{1}+1\right)I_{n_{1}}+D\left(G_{1}\right)\right)^{-1}A\left(\overline{G_{1}}\right)\right)\\
 & \cdot\left[1-\varGamma_{_{Q\left(G_{2}\right)}}\left(x-n_{1}\right)\varGamma_{_{Q\left(G_{1}\right)-D\left(G_{1}\right)+A\left(\overline{G_{1}}\right)\left(\left(x-n_{1}+1\right)I_{n_{1}}+D\left(G_{1}\right)\right)^{-1}A\left(\overline{G_{1}}\right)}}\left(x-n_{1}-n_{2}+1\right)\right].
\end{align*}
 }
\item [{\emph{(b)}}] \emph{
\begin{align*}
f_{_{Q\left(G_{1}\veebar G_{2}\right)}} & =det\left(\left(x-n_{1}\right)I_{n_{2}}-Q\left(G_{2}\right)\right)det\left(xI_{n_{1}}-D\left(G_{1}\right)\right)\\
 & \cdot det\left(\left(x-n_{2}\right)I_{n_{1}}-Q\left(G_{1}\right)-D\left(G_{1}\right)-A\left(G_{1}\right)\left(xI_{n_{1}}-D\left(G_{1}\right)\right)^{-1}A\left(G_{1}\right)\right)\\
 & \cdot\left[1-\varGamma_{_{Q\left(G_{2}\right)}}\left(x-n_{1}\right)\varGamma_{_{Q\left(G_{1}\right)+D\left(G_{1}\right)+A\left(G_{1}\right)\left(xI_{n_{1}}-D\left(G_{1}\right)\right)^{-1}A\left(G_{1}\right)}}\left(x-n_{2}\right)\right].
\end{align*}
}
\end{lyxlist}
\begin{proof}
The proof is similar to the proof of Theorem \textbf{\ref{thm: CHRAC. POLYNOMIAL Laplacian}}.
\end{proof}
\begin{cor}
\textup{\label{cor:ALQ1}}\textup{\emph{Let $F$ and $H$ be r-regular
non isomorphic cospectral graphs. Then for every $G$,}}
\end{cor}

\begin{lyxlist}{00.00.0000}
\item [{\emph{a)}}] \emph{$G\veebar F$ and $G\veebar H$ are $\left\{ A,Q,L\right\} $NICS.}
\item [{\emph{b)}}] \emph{$G\underset{=}{\lor}F$ and $G\underset{=}{\lor}H$
are $\left\{ A,Q,L\right\} $NICS.}
\end{lyxlist}
\begin{proof}
(a)\emph{ $G\veebar F$ }and\emph{ $G\veebar H$ }are non isomorphic
since $F$ and $H$ are non isomorphic. By Theorems \textbf{\ref{thm:4.1 Adjacency characteristic polynomial}},
\textbf{\ref{thm: CHRAC. POLYNOMIAL Laplacian}} and \textbf{\ref{thm: signless laplacian characteristic}},
$\ensuremath{f_{A(G\veebar F)}\left(x\right)}=\ensuremath{f_{A(G\veebar H)}\left(x\right)}$,
$f_{_{L\left(G\veebar F\right)}}\left(x\right)=f_{_{L\left(G\veebar H\right)}}\left(x\right)$
and $\ensuremath{f_{_{Q\left(G\veebar F\right)}}\left(x\right)}=\ensuremath{f_{_{Q\left(G\veebar H\right)}}\left(x\right)}$
since the matrices $A(F)$ and $A(H)$ have the same coronal (Lemma
\textbf{\ref{lem:()-If-M coronal}}) and the same characteristic polynomial.
This completes the proof of (a).

The proof of (b) is similar.\\
\end{proof}
\begin{cor}
\textup{\label{cor:Let-ALQ2}}\textup{\emph{Let $F$and $H$ be r-regular
non isomorphic cospectral graphs. Then for every $G$,}}
\end{cor}

\begin{lyxlist}{00.00.0000}
\item [{\emph{a)}}] \emph{$F\veebar G$ and $H\veebar G$ are $\left\{ A,Q,L\right\} $NICS.}
\item [{\emph{b)}}] \emph{$F\underset{=}{\lor}G$ and $H\underset{=}{\lor}G$
are $\left\{ A,Q,L\right\} $NICS.}
\end{lyxlist}
\begin{proof}
The proof is similar to the proof of Corollary \textbf{\ref{cor:ALQ1}}. 
\end{proof}
\begin{rem}
\emph{The following examples demonstrate the importance of the regularity
of the graphs $F$ and $H$.}
\end{rem}

\begin{example}
The graphs $F$ and $H$ in Figure \textbf{\ref{fig:Two non regular A-cospectral graphs-1}
}are non regular and\textbf{ $A$}-cospectral\textbf{\cite{cvetkovic1971graphs}}.
The joins $K_{2}\veebar F$ and $K_{2}\veebar H$ in Figure\textbf{
\ref{fig:K2NSF and K2NSH}} are not $A$-cospectral since the $A$-spectrum
of $K_{2}\veebar F$ is \{$-2.2332$, $-2$, $-1.618$, $0^{[3]}$,
$0.577$, $0.618$, $4.6562$\} and the $A$-spectrum of $K_{2}\veebar H$
is $\{-2.7039,-1.618,-1.2467,0^{[3]},0.2526,0.618,4.698\}$. The joins
$K_{2}\underset{=}{\vee}F$ and $K_{2}\underset{=}{\vee}H$ in Figure
\textbf{\ref{fig:K2NNSF and K2NNSH}} are not $A$-cospectral since
the $A$-spectrum of $K_{2}\underset{=}{\vee}F$ is \{$(-2)^{[2]}$,
$-1$, $0^{[4]}$, $0.4384$, $4.5616$\} and the $A$-spectrum of
$K_{2}\underset{=}{\vee}H$ is $\{-2.6056,(-1)^{[2[},0^{[5]},4.6056\}$.
The joins $F\veebar K_{2}$ and $H\veebar K_{2}$ in Figure \textbf{\ref{fig:FNSK2 and HNSK2}
}are not $A$-cospectral since the $A$-spectrum of $F\veebar K_{2}$
is \{$-3.2361$, $-2.5205$, $-1$, $-0.5812$, $0^{[5]}$, $1.0895$,
$1.2361$, $5.0122$\} and the $A$-spectrum of $H\veebar K_{2}$
is \{$-3.5337$, $-2.1915$, $-1$, $0^{[6]}$, $0.3034$, $1.3403$,
$5.0815\}$ and the joins $F\underset{=}{\lor}K_{2}$ and $H\underset{=}{\lor}K_{2}$
in Figure \textbf{\ref{fig:G1NNSK2 and G2NNSK2-1} }are not $A$-cospectral
since the $A$-spectrum of $F\underset{=}{\lor}K_{2}$ is \{$-3.1903$,
$-2.4142$, $-1.2946$, $(-1)^{[3]}$, $0.4046$, $0.4142$, $1^{[2]}$,
$1.8201$, $5.2602\}$ and the $A$-spectrum of $H\underset{=}{\lor}K_{2}$
is $\{-4.1337,(-1)^{[5]},0,0.8194,1^{[3]},5.3143\}$.
\end{example}

\begin{figure}[h]
\begin{centering}
\subfloat[$F$]{\includegraphics{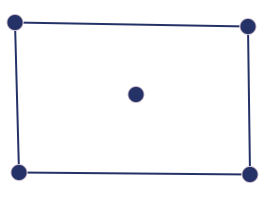}

\hspace*{0.1\textwidth}}\subfloat[$H$]{\includegraphics{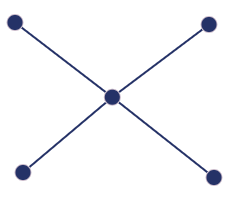}

}
\par\end{centering}
\caption{\label{fig:Two non regular A-cospectral graphs-1}Two non regular
$A$-cospectral graphs $F$ and $H$ .}
\end{figure}
\newpage{}

\begin{figure}[h]
\begin{centering}
\subfloat[$K_{2}\veebar F$]{\includegraphics[width=0.32\textwidth]{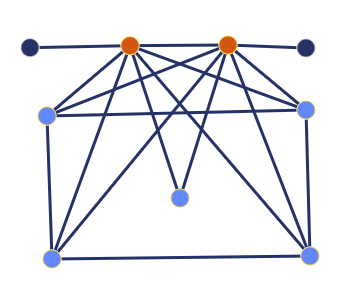}

\hspace*{0.1\textwidth}}\subfloat[$K_{2}\veebar H$]{\includegraphics[width=0.28\textwidth]{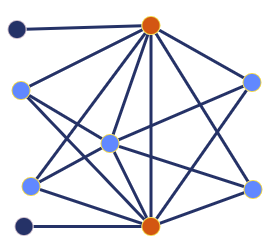}

}
\par\end{centering}
\caption{\label{fig:K2NSF and K2NSH}Two non $\textrm{A-cospectral}$ graphs
$\text{\ensuremath{K_{2}\veebar F}}$ and $K_{2}\veebar H.$}
\end{figure}

\begin{figure}[h]
\begin{centering}
\subfloat[$\text{\ensuremath{K_{2}\protect\underset{=}{\lor}}F}$]{\includegraphics[width=0.32\textwidth]{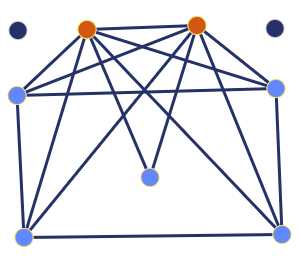}

\hspace*{0.1\textwidth}}\subfloat[$\text{\ensuremath{K_{2}\protect\underset{=}{\lor}}H}$]{\includegraphics[width=0.28\textwidth]{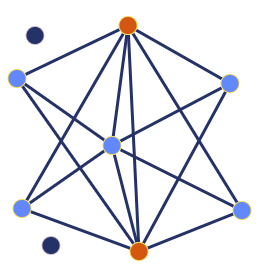}

}
\par\end{centering}
\caption{\label{fig:K2NNSF and K2NNSH}Two non $\textrm{A-cospectral}$ graphs
$\text{\ensuremath{K_{2}\protect\underset{=}{\lor}}F}$ and $\text{\ensuremath{K_{2}\protect\underset{=}{\lor}}H}$.}
\end{figure}

\begin{figure}[h]
\begin{centering}
\subfloat[$\text{F\ensuremath{\veebar K_{2}}}$]{\hspace*{0.06\textwidth}\includegraphics[width=0.32\textwidth]{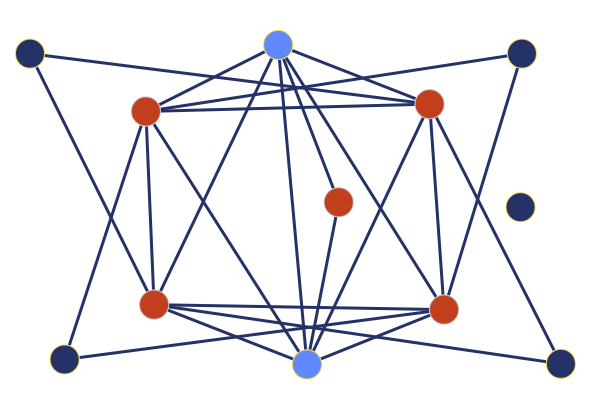}}\subfloat[$\text{H\ensuremath{\veebar K_{2}}}$]{\includegraphics[width=0.32\textwidth]{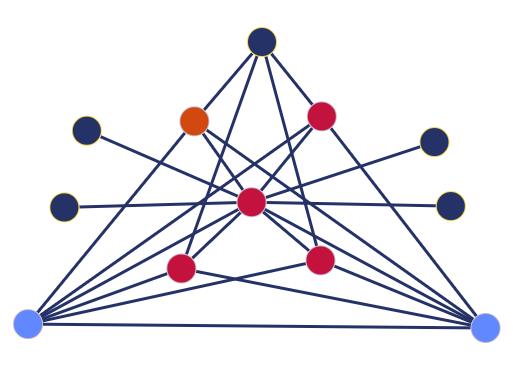}}
\par\end{centering}
\caption{\label{fig:FNSK2 and HNSK2}Two non $\textrm{A-cospectral}$ graphs
$\text{F\ensuremath{\veebar K_{2}}}$ and $\text{H\ensuremath{\veebar K_{2}}}.$}
\end{figure}

\begin{figure}[h]
\begin{centering}
\subfloat[$\text{F\ensuremath{\protect\underset{=}{\lor}K_{2}}}$]{\includegraphics[width=0.4\textwidth]{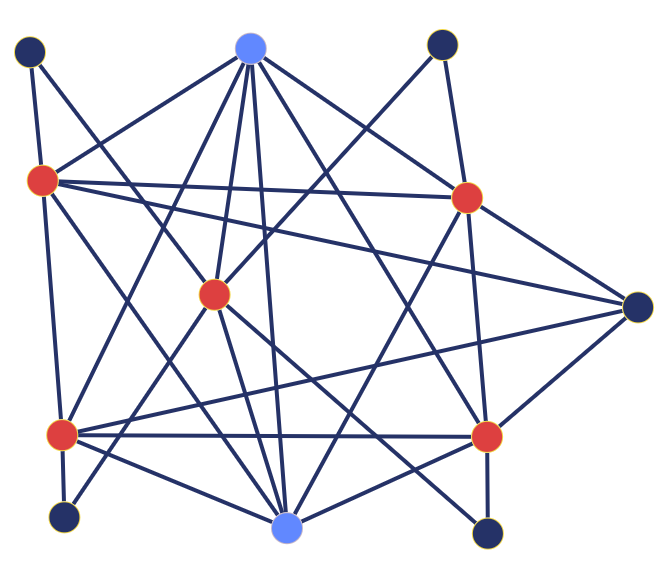}

}\subfloat[$\text{H\ensuremath{\protect\underset{=}{\lor}K_{2}}}$]{\includegraphics[width=0.4\textwidth]{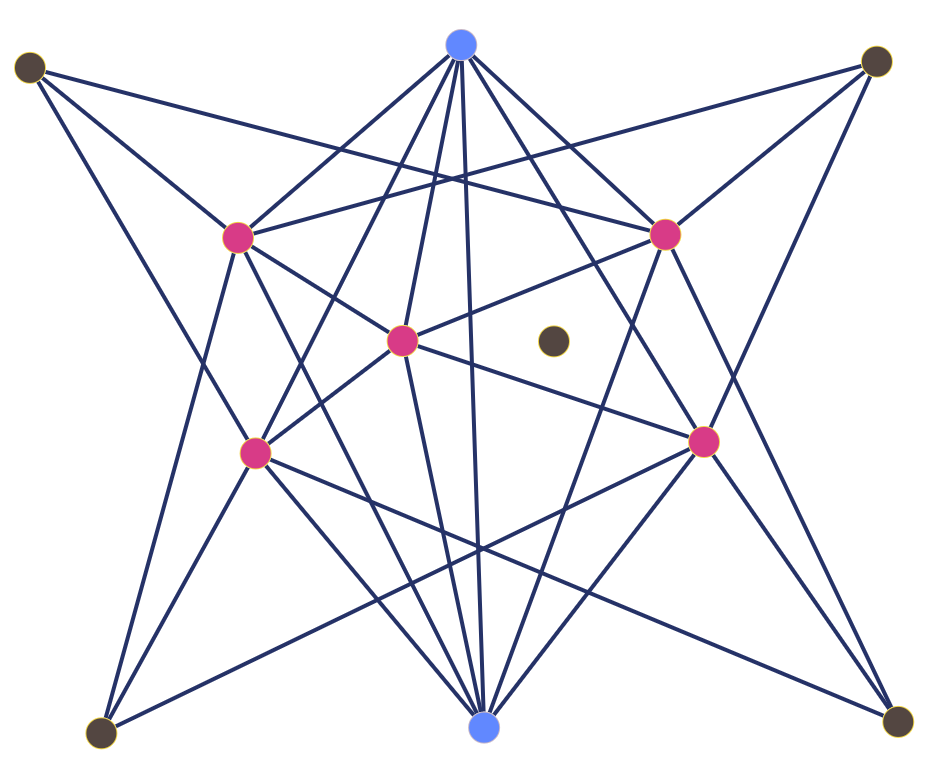}}
\par\end{centering}
\caption{\label{fig:G1NNSK2 and G2NNSK2-1}Two non $\textrm{A-cospectral}$
graphs $\text{F\ensuremath{\protect\underset{=}{\lor}K_{2}} }$and
$\text{H\ensuremath{\protect\underset{=}{\lor}K_{2}}}.$}
\end{figure}

\clearpage{}
\begin{rem}
\emph{Numerical computations suggest that in the corollaries }\textbf{\emph{\ref{cor:ALQ1}}}\emph{
and }\textbf{\emph{\ref{cor:Let-ALQ2}}}\emph{, $\left\{ A,Q,L\right\} $
can be replaced by $\left\{ A,Q,L,\mathcal{L}\right\} $. }
\end{rem}

\begin{conjecture}
\end{conjecture}

\begin{lyxlist}{00.00.0000}
\item [{\emph{a)}}] \emph{Let $H_{1}$ and $H_{2}$ be regular non isomorphic
cospectral graphs. Then for every $G$, $G\veebar H_{1}$ and $G\veebar H_{2}$
are $\left\{ A,Q,L,\mathcal{L}\right\} $NICS and $G\underset{=}{\lor}H_{1}$
and $G\underset{=}{\lor}H_{2}$ are $\left\{ A,Q,L,\mathcal{L}\right\} $NICS.}
\item [{\emph{b)}}] \emph{Let $G_{1}$ and $G_{2}$ be regular non isomorphic
cospectral graphs. Then for every $H$, $G_{1}\veebar H$ and $G_{2}\veebar H$
are $\left\{ A,Q,L,\mathcal{L}\right\} $NICS and $G_{1}\underset{=}{\lor}H$
and $G_{2}\underset{=}{\lor}H$ are $\left\{ A,Q,L,\mathcal{L}\right\} $NICS.}
\end{lyxlist}

\section{$\mathcal{L}-Spectra$ of NS Joins}

Let $G_{i},H_{i}$ be $r_{i}$-regular graphs, i = 1, 2. Lu, Ma and
Zhang showed that if $G_{1}$ and $H_{1}$ are cospectral and $G_{2}$
and $H_{2}$ are cospectral and non isomorphic, then $G_{1}\veebar G_{2}$
and $H_{1}\veebar H_{2}$ are $\left\{ A,L,Q\right\} $$\textrm{NICS}$.
In this section,we extend this result by showing that $G_{1}\veebar G_{2}$
and $H_{1}\veebar H_{2}$ are $\left\{ A,L,Q,\mathcal{L}\right\} $$\textrm{NICS}$.
To do it we determine the spectrum of the normalized Laplacian of
the graph $G_{1}\veebar G_{2}.$
\begin{thm}
\label{thm: about NS for NL}\textup{\emph{Let G$_{1}$ be a r$_{1}$-regular
graph with n$_{1}$ vertices and $G_{2}$ be a $r_{2}$-regular graph
with $n_{2}$vertices. Then the normalized Laplacian spectrum of $G_{1}\veebar G_{2}$
consists of:}}
\end{thm}

\begin{itemize}
\item \emph{$1+\frac{r_{2}\left(\delta_{i}\left(G_{2}\right)-1\right)}{n_{1}+r_{2}}$
for $i=2,3,...,n_{2};$}
\item \emph{$1+\frac{\left(\delta_{i}\left(G_{1}\right)-1\right)\left(\sqrt{9r_{1}^{2}+4r_{1}n_{2}}+r_{1}\right)}{2\left(2r_{1}+n_{2}\right)}$
for $i=2,3,...,n_{1};$}
\item \emph{$1+\frac{\left(1-\delta_{i}\left(G_{1}\right)\right)\left(\sqrt{9r_{1}^{2}+4r_{1}n_{2}}-r_{1}\right)}{2\left(2r_{1}+n_{2}\right)}$
for $i=2,3,...,n_{1};$}
\item \emph{the three roots of the equation 
\[
\left(2r_{1}r_{2}+2r_{1}n_{1}+n_{2}r_{2}+n_{1}n_{2}\right)x^{3}-\left(3r_{1}r_{2}+5r_{1}n_{1}+2r_{2}n_{2}+3n_{1}n_{2}\right)x^{2}+\left(3r_{1}n_{1}+n_{2}r_{2}+2n_{1}n_{2}\right)x=0.
\]
}
\end{itemize}
\begin{proof}
Let $u_{1},u_{2},\ldots,u_{n_{1}}$be the vertices of $G_{1}$, $u'_{1},u'_{2},\ldots,u'_{n_{1}}be$
the vertices added by the splitting and $v_{1},v_{2},...,v_{n_{2}}$
be the vertices of $G_{2}$. Under this vertex partitioning the adjacency
matrix of $G_{1}\veebar G_{2}$ is, 
\[
A\left(G_{1}\veebar G_{2}\right)=\left(\begin{array}{ccc}
A\left(G_{1}\right)\,\,\,\, & \,\,\,A\left(G_{1}\right)\,\,\,\,\, & J_{n_{1}\times n_{2}}\\
A\left(G_{1}\right) & O_{n_{1}\times n_{1}} & O_{n_{1}\times n_{2}}\\
J_{n_{2}\times n_{1}} & O_{n_{2}\times n_{1}} & A\left(G_{2}\right)
\end{array}\right)
\]
 The corresponding degrees matrix of $G_{1}\veebar G_{2}$ is, 
\[
D\left(G_{1}\veebar G_{2}\right)=\left(\begin{array}{ccc}
\left(2r_{1}+n_{2}\right)I_{n_{1}}\,\,\,\,\,\,\, & O_{n_{1}\times n_{1}} & O_{n_{1}\times n_{2}}\\
O_{n_{1}\times n_{1}} & r_{1}I_{n_{1}} & O_{n_{1}\times n_{2}}\\
O_{n_{2}\times n_{1}} & O_{n_{2}\times n_{1}} & \,\,\,\,\,\,\,\,\,\left(r_{2}+n_{1}\right)I_{n_{2}}
\end{array}\right).
\]

By simple calculation we get 
\[
\mathcal{L}\left(G_{1}\veebar G_{2}\right)=\left(\begin{array}{ccc}
I_{n_{1}}-\frac{A\left(G_{1}\right)}{2r_{1}+n_{2}}\,\,\,\,\,\,\,\, & \frac{-A\left(G_{1}\right)}{\sqrt{r_{1}\left(2r_{1}+n_{2}\right)}}\,\,\,\,\,\,\,\,\, & \frac{-J_{n_{1}\times n_{2}}}{\sqrt{\left(2r_{1}+n_{2}\right)\left(r_{2}+n_{1}\right)}}\\
\\
\frac{-A\left(G_{1}\right)}{\sqrt{r_{1}\left(2r_{1}+n_{2}\right)}} & I_{n_{1}} & O_{n_{1}\times n_{2}}\\
\\
\frac{-J_{n_{2}\times n_{1}}}{\sqrt{\left(2r_{1}+n_{2}\right)\left(r_{2}+n_{1}\right)}} & O_{n_{2}\times n_{1}} & I_{n_{2}}-\frac{A\left(G_{2}\right)}{r_{2}+n_{1}}
\end{array}\right).
\]

We prove the theorem by constructing an orthogonal basis of eigenvectors
of $\mathcal{L}\left(G_{1}\veebar G_{2}\right)$. Since $G_{2}$ is
$r_{2}$-regular, the vector $\mathbf{1}_{n_{2}}$is an eigenvector
of $A(G_{2})$ that corresponds to $\lambda_{1}\left(G_{2}\right)=r_{2}$.
For $i=2,3,\ldots,n_{2}$ let $Z_{i}$ be an eigenvector of $A\left(G_{2}\right)$
that corresponds to $\lambda_{i}\left(G_{2}\right)$. Then $\mathbf{1}_{n_{2}}^{T}Z_{i}=0$
and $\left(0_{1\times n_{1}},0_{1\times n_{1}},Z_{i}^{T}\right)^{T}$
is an eigenvector of $\mathcal{L}\left(G_{1}\veebar G_{2}\right)$
corresponding to the eigenvalue $1-\frac{\lambda_{i}\left(G_{2}\right)}{r_{2}+n_{1}}$
.

By Remark \textbf{\ref{rem:ALQl}}, $1+\frac{r_{2}\left(\delta_{i}\left(G_{2}\right)-1\right)}{n_{1}+r_{2}}$
are eigenvalues of $\mathcal{L}\left(G_{1}\veebar G_{2}\right)$ for
$i=2,...,n_{2}.$ 

For $i=2,...,n_{1}$, let $X_{i}$ be an eigenvector of $A(G_{1})$
corresponding to the eigenvalue $\lambda_{i}\left(G_{1}\right)$.
We now look for a non zero real number $\alpha$ such that $\left(\begin{array}{ccc}
X_{i}^{T}\,\,\,\, & \alpha X_{i}^{T}\,\,\, & 0_{1\times n_{2}}\end{array}\right)^{T}$is an eigenvector of $\mathcal{L}(G_{1}\veebar G_{2})$. 

\[
\mathcal{L}\left(\begin{array}{c}
X_{i}\\
\alpha X_{i}\\
0_{n_{2}\times1}
\end{array}\right)=\left(\begin{array}{c}
X_{i}-\frac{\lambda_{i}\left(G_{1}\right)}{2r_{1}+n_{2}}X_{i}-\frac{\lambda_{i}\left(G_{1}\right)\alpha}{\sqrt{r_{1}\left(2r_{1}+n_{2}\right)}}X_{i}\\
-\frac{\lambda_{i}\left(G_{1}\right)}{\sqrt{r_{1}\left(2r_{1}+n_{2}\right)}}X_{i}+\alpha X_{i}\\
0_{n_{2}\times1}
\end{array}\right)=\left(\begin{array}{c}
1-\frac{\lambda_{i}\left(G_{1}\right)}{2r_{1}+n_{2}}-\frac{\lambda_{i}\left(G_{1}\right)\alpha}{\sqrt{r_{1}\left(2r_{1}+n_{2}\right)}}\\
-\frac{\lambda_{i}\left(G_{1}\right)}{\alpha\sqrt{r_{1}\left(2r_{1}+n_{2}\right)}}+1\\
0_{n_{2}\times1}
\end{array}\right)\left(\begin{array}{c}
X_{i}\\
\alpha X_{i}\\
0_{n_{2}\times1}
\end{array}\right),
\]
so, $\alpha$ must be a root of the equation 
\begin{equation}
1-\frac{\lambda_{i}\left(G_{1}\right)}{2r_{1}+n_{2}}-\frac{\lambda_{i}\left(G_{1}\right)\alpha}{\sqrt{r_{1}\left(2r_{1}+n_{2}\right)}}=-\frac{\lambda_{i}\left(G_{1}\right)}{\alpha\sqrt{r_{1}\left(2r_{1}+n_{2}\right)}}+1,\label{eq:equation of NS about eigenvalues}
\end{equation}
\[
\sqrt{2r_{1}+n_{2}}\alpha^{2}+\sqrt{r_{1}}\alpha-\sqrt{2r_{1}+n_{2}}=0.
\]

Thus, $\alpha=\frac{2\sqrt{2r_{1}+n_{2}}}{\sqrt{9r_{1}+4n_{2}}+\sqrt{r_{1}}}$
or $\alpha=\frac{-2\sqrt{2r_{1}+n_{2}}}{\sqrt{9r_{1}+4n_{2}}-\sqrt{r_{1}}}$.
Substituting the values of $\alpha$ in the right side of\\
\\
 (\textbf{\ref{eq:equation of NS about eigenvalues}}), we get by
Remark \textbf{\ref{rem:ALQl}} that \\

$1+\frac{\left(\delta_{i}\left(G_{1}\right)-1\right)\left(\sqrt{9r_{1}^{2}+4r_{1}n_{2}}+r_{1}\right)}{n_{1}+r_{2}}$
, $1+\frac{\left(1-\delta_{i}\left(G_{1}\right)\right)\left(\sqrt{9r_{1}^{2}+4r_{1}n_{2}}-r_{1}\right)}{2\left(2r_{1}+n_{2}\right)}$
are eigenvalues of $\mathcal{L}\left(G_{1}\veebar G_{2}\right)$ for
$i=2,3,...,n_{1}.$

So far, we obtained $n_{2}-1+2\left(n_{1}-1\right)=2n_{1}+n_{2}-3$
eigenvalues of $\mathcal{L}\left(G_{1}\veebar G_{2}\right)$. Their
eigenvectors are orthogonal to $\left(\mathbf{1}_{n_{1}}^{T},0_{1\times n_{1}},0_{1\times n_{2}}\right)^{T},$$\left(0_{1\times n_{1}},1_{n_{1}}^{T},0_{1\times n_{2}}\right)^{T}$
and $\left(0_{1\times n_{1}},0_{1\times n_{1}},1_{n_{2}}^{T}\right)^{T}.$

To find three additional eigenvalues, we look for eigenvectors of
$\mathcal{L}\left(G_{1}\veebar G_{2}\right)$ of the form $Y=\left(\alpha\mathbf{1}_{n_{1}}^{T},\beta\mathbf{1}_{n_{1}}^{T},\gamma\mathbf{1}_{n_{2}}^{T}\right)^{T}$
for $\left(\alpha,\beta,\gamma\right)\neq\left(0,0,0\right)$. Let
$x$ be an eigenvalue of $\mathcal{L}\left(G\veebar G_{2}\right)$
corresponding to the eigenvector $Y$ . From $\mathcal{L}Y=xY$ we
get

\[
\begin{cases}
\alpha-\frac{r_{1}}{2r_{1}+n_{2}}\alpha\,-\frac{r_{1}}{\sqrt{r_{1}\left(2r_{1}+n_{2}\right)}}\beta\,-\,\frac{n_{2}}{\sqrt{\left(2r_{1}+n_{2}\right)\left(r_{2}+n_{1}\right)}}\gamma=\alpha x\\
\frac{-r_{1}}{\sqrt{r_{1}\left(2r_{1}+n_{2}\right)}}\alpha+\beta=\beta x\\
\frac{-n_{1}}{\sqrt{\left(2r_{1}+n_{2}\right)\left(r_{2}+n_{1}\right)}}\alpha+\gamma-\frac{r_{2}}{r_{2}+n_{1}}\gamma=\gamma x.
\end{cases}
\]

Thus
\[
\alpha-\frac{r_{1}}{2r_{1}+n_{2}}\alpha+\frac{r_{1}^{2}\alpha}{r_{1}(2r_{1}+n_{2})(x-1)}+\frac{n_{1}n_{2}(r_{2}+n_{1})\alpha}{(2r_{1}+n_{2})(r_{2}+n_{1})\left((x-1)(r_{2}+n_{1})+r_{2}\right)}=\alpha x.
\]
 Notice that $\alpha\neq0$, since if $\alpha=0$ then $\alpha=\beta=\gamma=0$
and also $x\neq1$ since $x=1$ implies that $\alpha=0$. 

Dividing by $\alpha$, we get the following cubic equation 
\[
\left(2r_{1}r_{2}+2r_{1}n_{1}+n_{2}r_{2}+n_{1}n_{2}\right)x^{3}-\left(3r_{1}r_{2}+5r_{1}n_{1}+2r_{2}n_{2}+3n_{1}n_{2}\right)x^{2}+\left(3r_{1}n_{1}+n_{2}r_{2}+2n_{1}n_{2}\right)x=0
\]
 and this completes the proof.
\end{proof}
Now we can answer Question \textbf{\ref{que:Construct-non-regular}}
by constructing pairs of non regular $\left\{ A,L,Q,\mathcal{L}\right\} $NICS
graphs.
\begin{cor}
\textup{\label{cor:corollary of NS}}\textup{\emph{Let $G_{i},H_{i}$
be $r_{i}$-regular graphs, i = 1, 2. If $G_{1}$ and $H_{1}$ are
cospectral and $G_{2}$ and $H_{2}$ are cospectral and non isomorphic
then $G_{1}\veebar G_{2}$ and $H_{1}\veebar H_{2}$ are $\left\{ A,L,Q,\mathcal{L}\right\} $NICS. }}
\end{cor}

\begin{proof}
\emph{$G_{1}\veebar G_{2}$ }and\emph{ $H_{1}\veebar H_{2}$ }are
non isomorphic since $G_{2}$ and $H_{2}$ are non isomorphic. By
Theorem \textbf{\ref{thm: about NS for NL}} and Theorems \textbf{3.1},
\textbf{3.2} and \textbf{3.3} in \textbf{\cite{lu2023spectra}}, the\emph{
}graphs\emph{ $H_{1}\veebar H_{2}$ }are\emph{ $\left\{ A,L,Q,\mathcal{L}\right\} $}NICS\emph{.}
\end{proof}
\begin{example}
Let $G_{1}=H_{1}=C_{4}$, and choose $G_{2}=G$ and $H_{2}=H$ where
$G$ and $H$ are graphs in Figure \textbf{\ref{fig:Two-regular-non isomorphic cospectral  10 vertices}},
then the graphs in Figure \textbf{\ref{fig:Two-regular-non isomorphic cospectral  10 vertices}}
are $\left\{ A,L,Q,\mathcal{L}\right\} $NICS. 
\end{example}

\begin{figure}[h]
\begin{centering}
\subfloat[$C_{4}\veebar G_{2}$]{\includegraphics[width=0.5\textwidth]{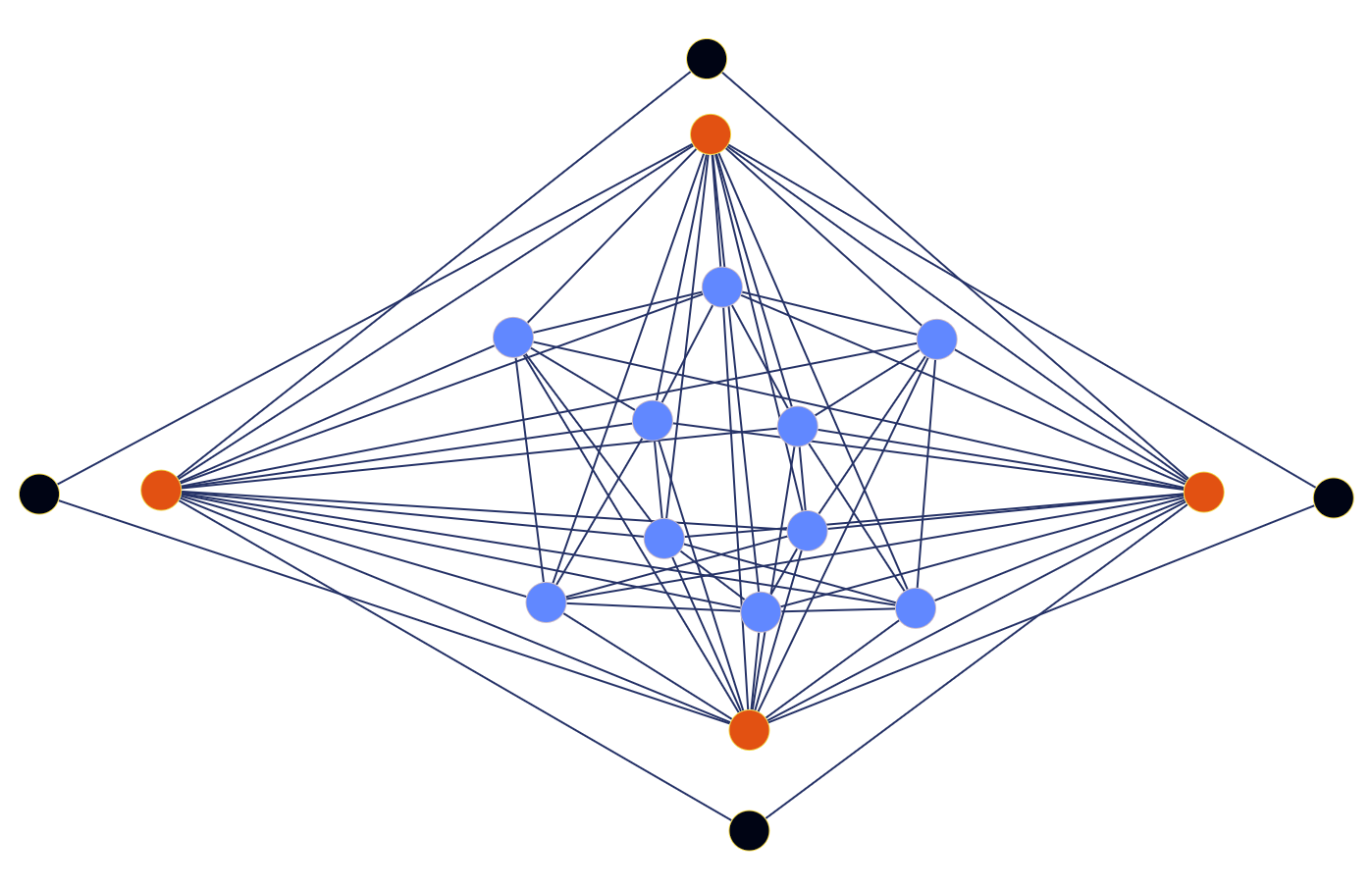}

\hspace{5bp}}\subfloat[$C_{4}\veebar H_{2}$ ]{\includegraphics[width=0.46\textwidth]{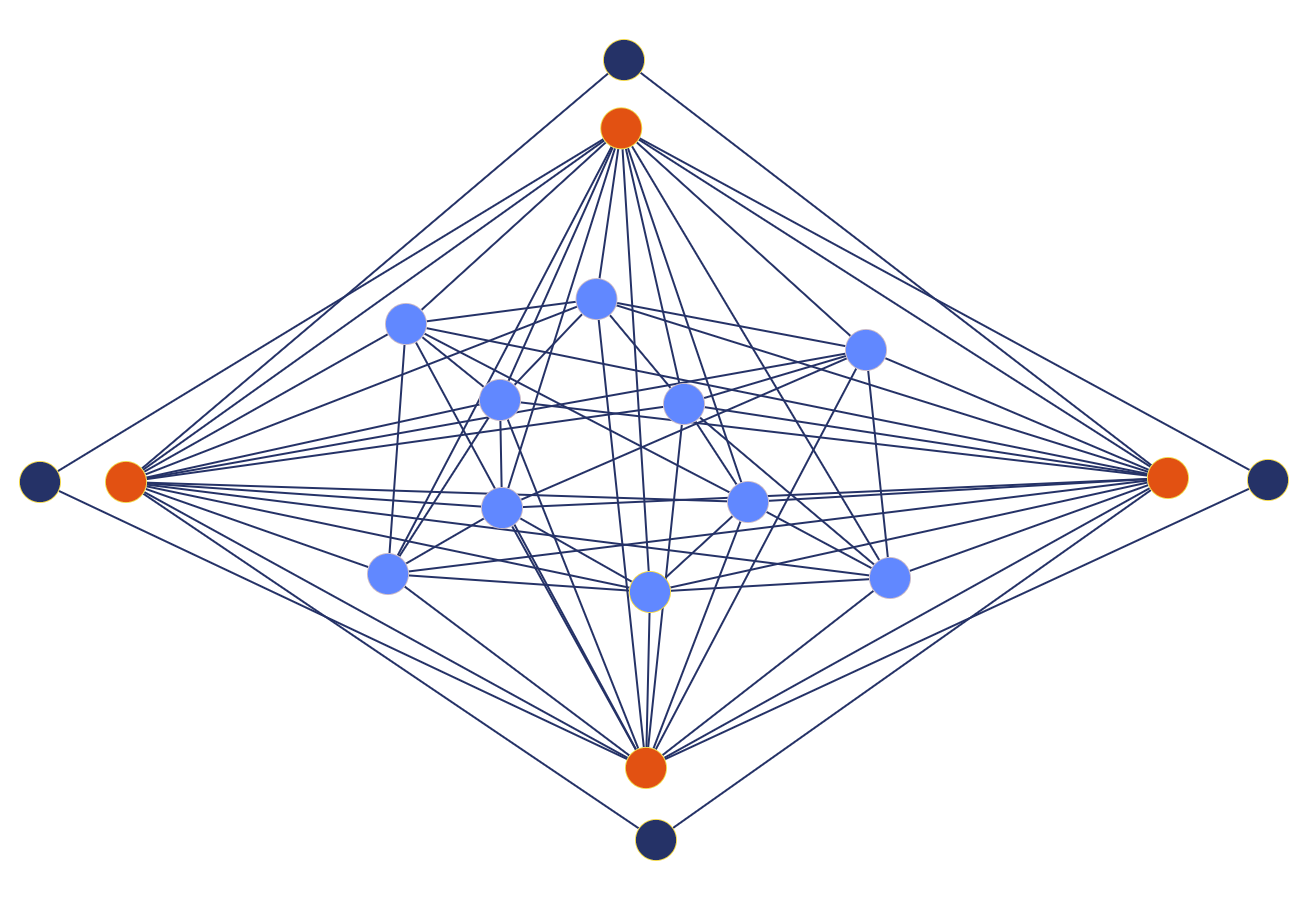}

}
\par\end{centering}
\caption{\label{fig:NS (GC4) and (HC4)}Non regular $\left\{ A,L,Q,\mathcal{L}\right\} $NICS
graphs. }
\end{figure}

\newpage{}

\section{Spectra of NNS Joins }

In this section we compute the $A$-spectrum, $L$-spectrum, $Q$-spectrum
and $\mathcal{L}$-spectrum of G$_{1}$$\underset{=}{\vee}$G$_{2}$
where $G_{1}$ and $G_{2}$ are regular.

We use it to answer Question \textbf{\ref{que:Construct-non-regular}}
by constructing pairs of non regular $\left\{ A,L,Q,\mathcal{L}\right\} $NICS
graphs.

\subsection{A-spectra of NNS join}

The adjacency matrix of G$_{1}$$\underset{=}{\vee}$ G$_{2}$ can
be written in a block form 

\begin{equation}
A\left(G_{1}\underset{=}{\vee}G_{2}\right)=\left(\begin{array}{ccc}
A\left(G_{1}\right) & \,\,\,\,A\left(\overline{G_{1}}\right) & \,\,\,\,\,J_{n_{1}\times n_{2}}\\
\\
A\left(\stackrel{\text{\_\_\_}}{G_{1}}\right) & \,\,\,\,\,\,O_{n_{1}\times n_{1}} & \,\,\,\,\,\,O_{n_{1}\times n_{2}}\\
\\
J_{n_{2}\times n_{1}} & \,\,\,\,\,\,O_{n_{2}\times n_{1}} & \,\,\,\,A\left(G_{2}\right)
\end{array}\right)\label{eq:2.1-1}
\end{equation}
 
\begin{thm}
\label{thm:A1}\textup{\emph{ Let G$_{1}$ be a r$_{1}$-regular graph
with n$_{1}$ vertices and $G_{2}$ be a $r_{2}$-regular graph with
$n_{2}$vertices. Then the adjacency spectrum of $G_{1}\underset{=}{\vee}G_{2}$$\,$consists
of:}}
\end{thm}

\begin{itemize}
\item \emph{(i) $\lambda_{j}\left(G_{2}\right)$ for each j = 2, 3, . .
. , n$_{2}$;}
\item \emph{(ii) two roots of the equation }\\
\emph{}\\
\emph{~~~~~~~~~~~~~~$x^{2}-\left(\lambda_{i}\left(G_{1}\right)\right)x-\left(\lambda_{i}^{2}\left(G_{1}\right)+2\lambda_{i}\left(G_{1}\right)+1\right)=0$
for each i = 2, 3, . . . , n$_{1}$;}
\item \emph{(iii) the three roots of the equation }\\
\emph{}\\
\emph{~~~~~~~~~~~~~$x^{3}-(r_{1}+r_{2})x^{2}+\left(r_{1}r_{2}-(n_{1}-r_{1}-1)^{2}-n_{1}n_{2}\right)x+r_{2}(n_{1}-r_{1}-1)^{2}=0$}
\end{itemize}
\begin{proof}
By Theorem \textbf{\ref{thm:4.1 Adjacency characteristic polynomial}},
the adjacency characteristic polynomial of $G_{1}\underset{=}{\lor}G_{2}$
is 

\begin{align*}
f_{A(G_{1}\underset{=}{\vee}G_{2})}(x) & =x^{n_{1}}f_{_{A\left(G_{2}\right)}}(x)det\left(xI_{n_{1}}-A(G_{1})-\frac{1}{x}A^{2}(\overline{G_{1}})\right)\left[1-\varGamma_{A(G_{2})}(x)\varGamma_{A(G_{1})+\frac{1}{x}A^{2}(\overline{G_{1}})}(x)\right]\\
 & =x^{n_{1}}\stackrel[j=1]{n_{2}}{\prod}\left(x-\lambda_{j}(G_{2})\right)det\left(xI_{n_{1}}-A(G_{1})-\frac{1}{x}A^{2}(\overline{G_{1}})\right)\left[1-\varGamma_{A(G_{2})}(x)\varGamma_{A(G_{1})+\frac{1}{x}A^{2}(\overline{G_{1}})}(x)\right]
\end{align*}
 Since $G_{1}$ and $G_{2}$ are regular, we can use Lemma \textbf{\ref{lem:()-If-M coronal}}
to get 

\begin{align*}
f_{A(G_{1}\underset{=}{\vee}G_{2})}(x) & =x^{n_{1}}\stackrel[j=1]{n_{2}}{\prod}\left(x-\lambda_{j}(G_{2})\right)det\left(xI_{n_{1}}-A(G_{1})-\frac{1}{x}A^{2}(\overline{G_{1}})\right)\left[1-\frac{n_{2}}{x-r_{2}}\,\frac{n_{1}}{x-r_{1}-\frac{1}{x}(n_{1}-r_{1}-1)^{2}}\right]\\
\\
 & =x^{n_{1}}\stackrel[j=1]{n_{2}}{\prod}\left(x-\lambda_{j}(G_{2})\right)det\left(xI_{n_{1}}-A(G_{1})-\frac{1}{x}(J-I-A(G_{1}))^{2}\right)\left[1-\frac{n_{1}n_{2}}{(x-r_{2})(x-r_{1}-\frac{1}{x}(n_{1}-r_{1}-1)^{2})}\right]\\
\\
 & =x^{n_{1}}\stackrel[j=1]{n_{2}}{\prod}\left(x-\lambda_{j}(G_{2})\right)det\left(xI_{n_{1}}-A(G_{1})-\frac{1}{x}(J^{2}-2J+I-2r_{1}J+2A(G_{1})+A^{2}(G_{1}))\right)\\
\\
 & \cdot\left[1-\frac{n_{1}n_{2}}{(x-r_{2})(x-r_{1}-\frac{1}{x}(n_{1}-r_{1}-1)^{2})}\right]\\
\\
 & =x^{n_{1}}\stackrel[j=1]{n_{2}}{\prod}\left(x-\lambda_{j}(G_{2})\right)\left[det\left(B\right)-\frac{1}{x}(n_{1}-2-2r_{1})1_{n_{1}}^{T}adj\left(B\right)1_{n_{1}}\right]\left[1-\frac{n_{1}n_{2}}{(x-r_{2})(x-r_{1}-\frac{1}{x}(n_{1}-r_{1}-1)^{2})}\right],
\end{align*}
 where $B=(x-\frac{1}{x})I_{n_{1}}-(1+\frac{2}{x})A(G_{1})-\frac{1}{x}A^{2}(G_{1}).$\\
Thus, based on Definition\textbf{ \ref{def:The-M-coronal-=000393(x)}
}and Lemma \textbf{\ref{lem:()-If-M coronal}}, we have
\begin{align*}
f_{A(G_{1}\underset{=}{\vee}G_{2})}(x)= & x^{n_{1}}\stackrel[j=1]{n_{2}}{\prod}\left(x-\lambda_{j}(G_{2})\right)\stackrel[i=1]{n_{1}}{\prod}\left[\left(x-\frac{1}{x}\right)-\left(1+\frac{2}{x}\right)\lambda_{i}(G_{1})-\frac{1}{x}\lambda_{i}^{2}(G_{1})\right]\\
\\
\cdot & \left[1-\frac{1}{x}(n_{1}-2-2r_{1})\varGamma_{\frac{1}{x}I_{n_{1}}+(1+\frac{2}{x})A(G_{1})+\frac{1}{x}A^{2}(G_{1})}(x)\right]\left[1-\frac{n_{1}n_{2}}{(x-r_{2})(x-r_{1}-\frac{1}{x}(n_{1}-r_{1}-1)^{2})}\right]\\
\\
= & x^{n_{1}}\stackrel[j=1]{n_{2}}{\prod}\left(x-\lambda_{j}(G_{2})\right)\stackrel[i=1]{n_{1}}{\prod}\left[\left(x-\frac{1}{x}\right)-\left(1+\frac{2}{x}\right)\lambda_{i}(G_{1})-\frac{1}{x}\lambda_{i}^{2}(G_{1})\right]\left[1-\frac{n_{1}-2-2r_{1}}{x}\cdot\frac{n_{1}}{x-\frac{1}{x}-r_{1}-\frac{2r_{1}}{x}-\frac{r_{1}^{2}}{x}}\right]\\
\\
\cdot & \left[1-\frac{n_{1}n_{2}}{(x-r_{2})(x-r_{1}-\frac{1}{x}(n_{1}-r_{1}-1)^{2})}\right]\\
\\
 & =\stackrel[j=1]{n_{2}}{\prod}\left(x-\lambda_{j}(G_{2})\right)x^{n_{1}}\stackrel[i=1]{n_{1}}{\prod}\left(\left(x-\frac{1}{x}\right)-\left(1+\frac{2}{x}\right)\lambda_{i}(G_{1})-\frac{1}{x}\lambda_{i}^{2}(G_{1})\right)\\
\\
 & \cdot\left(1-\frac{n_{1}(n_{1}-2-2r_{1})}{x^{2}-r_{1}x-r_{1}^{2}-2r_{1}-1}\right)\left(\frac{(x-r_{2})\left(x(x-r_{1})-(n_{1}-r_{1}-1)^{2}\right)-n_{1}n_{2}x}{(x-r_{2})\left(x(x-r_{1})-(n_{1}-r_{1}-1)^{2}\right)}\right)\\
\\
 & =\stackrel[j=2]{n_{2}}{\prod}\left(x-\lambda_{j}(G_{2})\right)\stackrel[i=2]{n_{1}}{\prod}\left(x^{2}-1-(x+2)\lambda_{i}(G_{1})-\lambda_{i}^{2}(G_{1})\right)\\
\\
 & \cdot\left(x(x-r_{1})-(n_{1}-r_{1}-1)^{2}\right)\left(\frac{(x-r_{2})\left(x(x-r_{1})-(n_{1}-r_{1}-1)^{2}\right)-n_{1}n_{2}x}{(x-r_{2})\left(x(x-r_{1})-(n_{1}-r_{1}-1)^{2}\right)-n_{1}n_{2}x}\right)\\
\\
 & =\stackrel[j=2]{n_{2}}{\prod}\left(x-\lambda_{j}(G_{2})\right)\stackrel[i=2]{n_{1}}{\prod}\left(x^{2}-1-(x+2)\lambda_{i}(G_{1})-\lambda_{i}^{2}(G_{1})\right)\\
\\
 & \cdot\left((x-r_{2})\left(x(x-r_{1})-(n_{1}-r_{1}-1)^{2}\right)-n_{1}n_{2}x\right)\\
\\
 & =\stackrel[j=2]{n_{2}}{\prod}\left(x-\lambda_{j}(G_{2})\right)\stackrel[i=2]{n_{1}}{\prod}\left(x^{2}-(\lambda_{i}(G_{1}))x-\left(\lambda_{i}^{2}(G_{1})+2\lambda_{i}(G_{1})+1\right)\right)\\
\\
 & \cdot\left(x^{3}-(r_{1}+r_{2})x^{2}+(r_{1}r_{2}-(n_{1}-r_{1}-1)^{2}-n_{1}n_{2})x+r_{2}(n_{1}-r_{1}-1)^{2}\right).
\end{align*}
 
\end{proof}

\subsection{L-spectra of NNS join}

The degrees of the vertices of $G_{1}\underset{=}{\vee}G_{2}$ are:
\begin{center}
d$_{G_{1}\underset{=}{\vee}G_{2}}$ ($u_{i}$) = n$_{1}+n_{2}$$-1$,
i = 1, . . . , n$_{1}$,
\par\end{center}

\begin{center}
d$_{G_{1}\underset{=}{\vee}G_{2}}$($u$$_{i}'$) = n$_{1}-r_{1}$$-1$,
i = 1, . . . , n$_{1}$,
\par\end{center}

\begin{center}
d$_{G_{1}\underset{=}{\vee}G_{2}}$ ($v_{j}$) = r$_{2}$ + n$_{1}$,
j = 1, . . . , n$_{2}$, 
\par\end{center}

so the degrees matrix of $G_{1}\underset{=}{\vee}G_{2}$, that corresponds
to (\textbf{\ref{eq:2.1-1}}) is
\begin{center}
\begin{equation}
D\left(G_{1}\underset{=}{\vee}G_{2}\right)=\left(\begin{array}{ccc}
(n_{1}+n_{2}-1)I_{n_{1}} & O_{n_{1}\times n_{1}} & \,O_{n_{1}\times n_{2}}\\
\\
O_{n_{1}\times n_{1}} & (n_{1}-r_{1}-1)I_{n_{1}} & \,O\\
\\
O_{n_{2}\times n_{1}} & O_{n_{2}\times n_{1}} & \,\,(r_{2}+n_{1})I_{n_{2}}
\end{array}\right)
\end{equation}
\\
\par\end{center}
\begin{thm}
\label{thm:L1}\textup{\emph{Let G$_{1}$ be a r$_{1}$-regular graph
with n$_{1}$ vertices and $G_{2}$ be a $r_{2}$-regular graph with
$n_{2}$vertices. Then the Laplacian spectrum of $G_{1}\underset{=}{\lor}G_{2}$
consists of:}}
\end{thm}

\begin{itemize}
\item \emph{$n_{1}+\mu_{j}(G_{2})$ for each $j=2,3,...,n_{2}$; }
\item \emph{two roots of the equation
\[
x^{2}+(2r_{1}-2n_{1}-n_{2}-\mu_{i}(G_{1})+2)x+n_{1}^{2}-2r_{1}n_{1}-2n_{1}+n_{1}n_{2}-r_{1}n_{2}-n_{2}+\mu_{i}(G_{1})(n_{1}+r_{1}+1-\mu_{i}(G_{1}))=0
\]
  for each $i=2,3,...,n_{1}$; }
\item \emph{the three roots of the equation
\[
x^{3}+\left(2r_{1}-3n_{1}-n_{2}+2\right)x^{2}+\left(n_{1}n_{2}-n_{2}r_{1}-n_{2}+2n_{1}^{2}-2r_{1}n_{1}-2n_{1}\right)x=0.
\]
}
\end{itemize}
\begin{proof}
By substituting $D\left(G_{1}\right)=r_{1}I_{n_{1}}$ in Theorem \textbf{\ref{thm: CHRAC. POLYNOMIAL Laplacian}},
the Laplacian characteristic polynomial of $G_{1}\underset{=}{\lor}G_{2}$
is

\begin{align*}
f_{_{L\left(G_{1}\underset{=}{\vee}G_{2}\right)}}\left(x\right) & =det\left(\left(x-n_{1}\right)I_{n_{2}}-L\left(G_{2}\right)\right)det\left(x-n_{1}+1+r_{1}\right)I_{n_{1}}\\
\\
 & \cdot det\left(\left(x-n_{1}-n_{2}+r_{1}+1\right)I_{n_{1}}-L\left(G_{1}\right)-\frac{1}{x-n_{1}+r_{1}+1}A^{2}\left(\overline{G_{1}}\right)\right)\\
\\
 & \cdot\left[1-\varGamma_{_{L\left(G_{2}\right)}}\left(x-n_{1}\right)\varGamma_{_{L\left(G_{1}\right)+\frac{1}{x-n_{1}+1+r_{1}}A^{2}\left(\overline{G_{1}}\right)}}\left(x-n_{1}-n_{2}+1+r_{1}\right)\right].
\end{align*}
 Using Lemma \textbf{\ref{lem:()-If-M coronal}}, we obtain 

\begin{align*}
f_{_{L\left(G_{1}\underset{=}{\vee}G_{2}\right)}}\left(x\right) & =det\left(\left(x-n_{1}\right)I_{n_{2}}-L\left(G_{2}\right)\right)\left(x-n_{1}+r_{1}+1\right)^{n_{1}}\\
\\
 & \cdot det\left(\left(x-n_{1}-n_{2}+r_{1}+1\right)I_{n_{1}}-L\left(G_{1}\right)-\frac{1}{x-n_{1}+r_{1}+1}A^{2}\left(\overline{G_{1}}\right)\right)\\
\\
 & \cdot\left[1-\frac{n_{2}n_{1}}{\left(x-n_{1}\right)\left(x-n_{1}-n_{2}+1+r_{1}-\frac{\left(n_{1}-r_{1}-1\right)^{2}}{x-n_{1}+r_{1}+1}\right)}\right]\\
\\
 & =\stackrel[j=1]{n_{2}}{\prod}\left(x-n_{1}-\mu_{j}\left(G_{2}\right)\right)\left(x-n_{1}+r_{1}+1\right)^{n_{1}}\left(\left(x-n_{1}-n_{2}+1+r_{1}\right)-\frac{\left(n_{1}-r_{1}-1\right)^{2}}{x-n_{1}+r_{1}+1}\right)\\
\\
 & \cdot\stackrel[i=2]{n_{1}}{\prod}\left(x-n_{1}-n_{2}+1+r_{1}-\mu_{i}\left(G_{1}\right)-\frac{\left(\mu_{i}\left(G_{1}\right)-r_{1}-1\right)^{2}}{x-n_{1}+r_{1}+1}\right)\\
\\
 & \cdot\left[1-\frac{n_{2}n_{1}}{\left(x-n_{1}\right)\left(x-n_{1}-n_{2}+1+r_{1}-\frac{\left(n_{1}-r_{1}-1\right)^{2}}{x-n_{1}+r_{1}+1}\right)}\right]\\
\\
 & =\stackrel[j=2]{n_{2}}{\prod}\left(x-n_{1}-\mu_{j}\left(G_{2}\right)\right)\stackrel[i=2]{n_{1}}{\prod}\left(\right.x^{2}+\left(2r_{1}-2n_{1}-n_{2}-\mu_{i}\left(G_{1}\right)+2\right)x+n_{1}^{2}-2r_{1}n_{1}\\
\\
 & -2n_{1}+n_{1}n_{2}-r_{1}n_{2}-n_{2}+\mu_{i}\left(G_{1}\right)\left(n_{1}+r_{1}+1-\mu_{i}\left(G_{1}\right)\right)\left.\right)\\
\\
 & \cdot\left[x^{3}+\left(2r_{1}-3n_{1}-n_{2}+2\right)x^{2}+\left(n_{1}n_{2}-n_{2}r_{1}-n_{2}+2n_{1}^{2}-2r_{1}n_{1}-2n_{1}\right)x\right].
\end{align*}

This completes the proof.
\end{proof}

\subsection{Q-spectra of NNS join}

\begin{thm}
\label{thm:Q1}\textup{\emph{Let G$_{1}$ be a r$_{1}$-regular graph
with n$_{1}$ vertices and $G_{2}$ be a $r_{2}$-regular graph with
$n_{2}$vertices.Then the signless Laplacian spectrum of G$_{1}$$\underset{=}{\vee}$
G$_{2}$ consists of:}}
\begin{itemize}
\item \emph{$n_{1}+\nu_{j}(G_{2})$ for each $j=2,3,...,n_{2};$}
\item \textup{\emph{two roots of the equation}}\emph{ $x^{2}+(2r_{1}-2n_{1}-n_{2}-\nu_{i}(G_{1})+2)x+n_{1}^{2}-2r_{1}n_{1}-2n_{1}+n_{1}n_{2}-r_{1}n_{2}-n_{2}+4r_{1}+\nu_{i}(G_{1})(n_{1}+r_{1}-3-\nu_{i}(G_{1}))$=0
for each $i=2,3,...,n_{1};$ }
\item \textup{\emph{the three roots of the equation }}\emph{
\begin{multline*}
x^{3}+\left(2-3n_{1}-n_{2}-2r_{2}\right)x^{2}+\left(n_{1}n_{2}-n_{2}r_{1}-n_{2}+2n_{1}^{2}+2r_{1}n_{1}-2n_{1}-2r_{1}-2r_{1}^{2}+4r_{2}n_{1}-4r_{2}+2r_{2}n_{2}\right)x\\
+2n_{1}r_{1}^{2}+2r_{1}n_{1}-2r_{1}n_{1}^{2}-2r_{2}n_{1}n_{2}+2r_{1}r_{2}n_{2}+2r_{2}n_{2}+4r_{2}r_{1}^{2}+4r_{1}r_{2}-4r_{1}r_{2}n_{1}=0
\end{multline*}
}
\end{itemize}
\begin{proof}
The proof is similar to the proof of Theorem \textbf{\ref{thm:L1}}.
\end{proof}
\end{thm}

\subsection{$\mathcal{L}$-spectra of NNS join}

Let $G_{1}$ be a $r_{1}$-regular graph on order $n_{1}$. Let $S$
be a subset of $\left\{ 2,3,\ldots,n_{1}\right\} $ such that $\delta_{i}\left(G_{1}\right)=1+\frac{1}{r_{1}}$
for $i\in S$, and denote the cardinality of $S$ by $n(S)$. Let
$G_{2}$ be a $r_{2}$-regular graph on order $n_{2}$. In the following
theorem we determine the normalized Laplacian spectrum of $G_{1}\underset{=}{\vee}G_{2}$
in terms of the normalized Laplacian eigenvalues of $G_{1}$ and $G_{2}.$
The proof is slightly more complicated than the proof of Theorem \textbf{\ref{thm: about NS for NL}
}and we consider three cases.
\begin{thm}
\textup{\label{thm: normalized Laplacian  of NNS for regular graph}}\textup{\emph{ }}
\end{thm}

\begin{lyxlist}{00.00.0000}
\item [{\emph{a)}}] \emph{If $S=\varPhi$, then the normalized Laplacian
spectrum of $G_{1}\underset{=}{\vee}G_{2}$ consists of:}\\
\emph{}\\
\emph{i) $1+\frac{r_{2}\left(\delta_{i}\left(G_{2}\right)-1\right)}{n_{1}+r_{2}}$,
for each $i=2,3,...,n_{2}$;}\\
\emph{}\\
\emph{ii) $1+\frac{2\left(1+r_{1}-r_{1}\delta_{i}\left(G_{1}\right)\right)^{2}}{r_{1}\left(1-\delta_{i}\left(G_{1}\right)\right)\left(n_{1}-r_{1}-1\right)\mp\sqrt{\left(n_{1}-r_{1}-1\right)\left[r_{1}^{2}\left(1-\delta_{i}\left(G_{1}\right)\right)^{2}\left(n_{1}-r_{1}-1\right)+4\left(1+r_{1}-r_{1}\delta_{i}\left(G_{1}\right)\right)^{2}\left(n_{1}+n_{2}-1\right)\right]}}$
, for each $\text{i=2,3,...,\ensuremath{n_{1}};}$}\\
\emph{}\\
\emph{iii) the three roots of the equation 
\begin{align*}
\,\,\,\,\,\,\,\, & \left(n_{1}^{2}+n_{1}n_{2}-n_{1}+r_{2}n_{1}+r_{2}n_{2}-r_{2}\right)x^{3}-\left(3n_{1}^{2}+3n_{1}n_{2}-3n_{1}-r_{1}n_{1}+2r_{2}n_{1}+2r_{2}n_{2}-2r_{2}-r_{1}r_{2}\right)x^{2}\\
 & +\left(2n_{1}^{2}+2n_{1}n_{2}-2n_{1}-r_{1}n_{1}+r_{2}n_{2}\right)x=0.
\end{align*}
}
\item [{\emph{b)}}] \emph{If $S=\left\{ 2,3,\ldots,n_{1}\right\} $, then
the normalized Laplacian spectrum of $G_{1}\underset{=}{\vee}G_{2}$
consists of:}\\
\emph{}\\
\emph{i) $1+\frac{r_{2}\left(\delta_{i}\left(G_{2}\right)-1\right)}{n_{1}+r_{2}}$
for each $i=2,3,...,n_{2};$}\\
\emph{}\\
\emph{ii) $(1+\frac{1}{n_{1}+n_{2}-1})^{[n_{1}-1]}$ , $0^{[n_{1}+1]}$
and $\frac{n_{1}^{2}+2n_{1}n_{2}+r_{2}n_{2}-n_{1}}{(r_{2}+n_{1})(n_{1}+n_{2}-1)}$
.}
\item [{\emph{c)}}] \emph{If $S$$\neq\varPhi$ and $S\neq\left\{ 2,3,\ldots,n_{1}\right\} $,
then the normalized Laplacian spectrum of $G_{1}\underset{=}{\vee}G_{2}$
consists of: }\\
\emph{}\\
\emph{i) $1+\frac{r_{2}\left(\delta_{i}\left(G_{2}\right)-1\right)}{n_{1}+r_{2}}$
for each $i=2,3,...,n_{2};$}\\
\emph{}\\
\emph{ii) $1+\frac{2\left(1+r_{1}-r_{1}\delta_{i}\left(G_{1}\right)\right)^{2}}{r_{1}\left(1-\delta_{i}\left(G_{1}\right)\right)\left(n_{1}-r_{1}-1\right)\mp\sqrt{\left(n_{1}-r_{1}-1\right)\left[r_{1}^{2}\left(1-\delta_{i}\left(G_{1}\right)\right)^{2}\left(n_{1}-r_{1}-1\right)+4\left(1+r_{1}-r_{1}\delta_{i}\left(G_{1}\right)\right)^{2}\left(n_{1}+n_{2}-1\right)\right]}}$
for each}\\
\emph{}\\
\emph{ $\,\,\,\,\,\,\,\,i\in\left\{ 2,3,\ldots,n_{1}\right\} \setminus S$;}\\
\emph{}\\
\emph{iii) $1^{[n(S)]}$ and $(1+\frac{1}{n_{1}+n_{2}-1})^{[n(S)]}$.}\\
\emph{}\\
\emph{iv) the three roots of the equation 
\begin{align*}
\,\,\,\,\,\,\,\, & \left(n_{1}^{2}+n_{1}n_{2}-n_{1}+r_{2}n_{1}+r_{2}n_{2}-r_{2}\right)x^{3}-\left(3n_{1}^{2}+3n_{1}n_{2}-3n_{1}-r_{1}n_{1}+2r_{2}n_{1}+2r_{2}n_{2}-2r_{2}-r_{1}r_{2}\right)x^{2}\\
 & +\left(2n_{1}^{2}+2n_{1}n_{2}-2n_{1}-r_{1}n_{1}+r_{2}n_{2}\right)x=0.
\end{align*}
 }
\end{lyxlist}
\begin{proof}
\textbf{(a)} If $S=\varPhi$, then $\delta_{i}\left(G_{1}\right)\neq1+\frac{1}{r_{1}}$
for each $i=2,3,\ldots,n_{1}$, so $\lambda_{i}\left(G_{1}\right)\neq-1$
for each $i=2,3,\ldots,n_{1}$. The normalized Laplacian matrix of
$G_{1}\underset{=}{\vee}G_{2}$ is: 

\[
\mathcal{L}\left(G_{1}\underset{=}{\vee}G_{2}\right)=\left(\begin{array}{ccc}
I_{n_{1}}-\frac{A\left(G_{1}\right)}{n_{1}+n_{2}-1}\,\,\,\,\,\,\,\, & \frac{-A\left(\textrm{\ensuremath{\overline{G_{1}}}}\right)}{\sqrt{\left(n_{1}+n_{2}-1\right)\left(n_{1}-r_{1}-1\right)}}\,\,\,\,\,\,\,\,\, & \frac{-J_{n_{1}\times n_{2}}}{\sqrt{\left(n_{1}+n_{2}-1\right)\left(r_{2}+n_{1}\right)}}\\
\\
\frac{-A\left(\overline{G_{1}}\right)}{\sqrt{\left(n_{1}+n_{2}-1\right)\left(n_{1}-r_{1}-1\right)}} & I_{n_{1}} & O_{n_{1}\times n_{2}}\\
\\
\frac{-J_{n_{2}\times n_{1}}}{\sqrt{\left(n_{1}+n_{2}-1\right)\left(r_{2}+n_{1}\right)}} & O_{n_{2}\times n_{1}} & I_{n_{2}}-\frac{A\left(G_{2}\right)}{r_{2}+n_{1}}
\end{array}\right).
\]

Since $G_{2}$ is $r_{2}$-regular, the vector $\mathbf{1}_{n_{2}}$is
an eigenvector of $A(G_{2})$ that corresponds to $\lambda_{1}\left(G_{2}\right)=r_{2}$.
For $i=2,3,\ldots,n_{2}$ let $Z_{i}$ be an eigenvector of $A\left(G_{2}\right)$
that corresponds to $\lambda_{i}\left(G_{2}\right)$. Then $\mathbf{1}_{n_{2}}^{T}Z_{i}=0$
and $\left(0_{1\times n_{1}},0_{1\times n_{1}},Z_{i}^{T}\right)^{T}$
is an eigenvector of $\mathcal{L}\left(G_{1}\underset{=}{\vee}G_{2}\right)$
corresponding to the eigenvalue $1-\frac{\lambda_{i}\left(G_{2}\right)}{r_{2}+n_{1}}$
. By Remark \textbf{\ref{rem:ALQl}}, $1+\frac{r_{2}\left(\delta_{i}\left(G_{2}\right)-1\right)}{n_{1}+r_{2}}$
are an eigenvalues of $\landupintop\landupintop$$\mathcal{L}\left(G_{1}\underset{=}{\lor}G_{2}\right)$
for $i=2,...,n_{2}.$

For $i=2,...,n_{1}$, let $X_{i}$ be an eigenvector of $A(G_{1})$
corresponding to the eigenvalue $\lambda_{i}\left(G_{1}\right)$.
We now look for a non zero real number $\alpha$ such that $\left(\begin{array}{ccc}
X_{i}^{T}\,\,\,\, & \alpha X_{i}^{T}\,\,\, & 0_{1\times n_{2}}\end{array}\right)^{T}$is an eigenvector of $\mathcal{L}(G_{1}\underset{=}{\vee}G_{2})$.
Notice that $\alpha\neq0$, since if $\alpha=0$ then $\lambda_{i}\left(G_{1}\right)=-1$.
\begin{align*}
\mathcal{L}\left(\begin{array}{c}
X_{i}\\
\alpha X_{i}\\
0_{n_{1}\times1}
\end{array}\right) & =\left(\begin{array}{c}
X_{i}-\frac{\lambda_{i}\left(G_{1}\right)}{n_{1}+n_{2}-1}X_{i}+\frac{1+\lambda_{i}\left(G_{1}\right)}{\sqrt{\left(n_{1}+n_{2}-1\right)\left(n_{1}-r_{1}-1\right)}}\alpha X_{i}\\
\\
\frac{1+\lambda_{i}\left(G_{1}\right)}{\sqrt{\left(n_{1}+n_{2}-1\right)\left(n_{1}-r_{1}-1\right)}}X_{i}+\alpha X_{i}\\
\\
0_{n_{1}\times1}
\end{array}\right)\\
\\
 & =\left(\begin{array}{c}
1-\frac{\lambda_{i}\left(G_{1}\right)}{n_{1}+n_{2}-1}+\frac{1+\lambda_{i}\left(G_{1}\right)}{\sqrt{\left(n_{1}+n_{2}-1\right)\left(n_{1}-r_{1}-1\right)}}\alpha\\
\\
\frac{1+\lambda_{i}\left(G_{1}\right)}{\alpha\sqrt{\left(n_{1}+n_{2}-1\right)\left(n_{1}-r_{1}-1\right)}}+1\\
\\
0_{n_{1}\times1}
\end{array}\right)\left(\begin{array}{c}
X_{i}\\
\alpha X_{i}\\
0_{n_{1}\times1}
\end{array}\right).
\end{align*}
 Thus

\begin{equation}
1-\frac{\lambda_{i}\left(G_{1}\right)}{n_{1}+n_{2}-1}+\frac{1+\lambda_{i}\left(G_{1}\right)}{\sqrt{\left(n_{1}+n_{2}-1\right)\left(n_{1}-r_{1}-1\right)}}\alpha=\frac{1+\lambda_{i}\left(G_{1}\right)}{\alpha\sqrt{\left(n_{1}+n_{2}-1\right)\left(n_{1}-r_{1}-1\right)}}+1\label{eq:NNS OF N}
\end{equation}

\[
-\frac{\lambda_{i}\left(G_{1}\right)}{n_{1}+n_{2}-1}+\frac{1+\lambda_{i}\left(G_{1}\right)}{\sqrt{\left(n_{1}+n_{2}-1\right)\left(n_{1}-r_{1}-1\right)}}\alpha=\frac{1+\lambda_{i}\left(G_{1}\right)}{\alpha\sqrt{\left(n_{1}+n_{2}-1\right)\left(n_{1}-r_{1}-1\right)}}
\]

\[
\frac{\alpha^{2}\left(1+\lambda_{i}\left(G_{1}\right)\right)-\left(1+\lambda_{i}\left(G_{1}\right)\right)}{\alpha\sqrt{\left(n_{1}+n_{2}-1\right)\left(n_{1}-r_{1}-1\right)}}=\frac{\lambda_{i}\left(G_{1}\right)}{n_{1}+n_{2}-1}
\]

\[
\left(1+\lambda_{i}\left(G_{1}\right)\right)\sqrt{n_{1}+n_{2}-1}\alpha^{2}-\lambda_{i}\left(G_{1}\right)\sqrt{n_{1}-r_{1}-1}\alpha-\left(1+\lambda_{i}\left(G_{1}\right)\right)\sqrt{n_{1}+n_{2}-1}=0,
\]
 so

\[
\alpha_{1,2}=\frac{\lambda_{i}\left(G_{1}\right)\sqrt{n_{1}-r_{1}-1}}{2\left(1+\lambda_{i}\left(G_{1}\right)\right)\sqrt{n_{1}+n_{2}-1}}\mp\sqrt{\frac{\lambda_{i}\left(G_{1}\right)^{2}\left(n_{1}-r_{1}-1\right)}{4\left(1+\lambda_{i}\left(G_{1}\right)\right)^{2}\left(n_{1}+n_{2}-1\right)}+1.}
\]

Substituting the values of $\alpha$ in the right side of (\textbf{\ref{eq:NNS OF N}}),
we get by Remark \textbf{\ref{rem:ALQl}} that 

$1+\frac{2\left(1+r_{1}-r_{1}\delta_{i}\left(G_{1}\right)\right)^{2}}{r_{1}\left(1-\delta_{i}\left(G_{1}\right)\right)\left(n_{1}-r_{1}-1\right)\mp\sqrt{\left(n_{1}-r_{1}-1\right)\left[r_{1}^{2}\left(1-\delta_{i}\left(G_{1}\right)\right)^{2}\left(n_{1}-r_{1}-1\right)+4\left(1+r_{1}-r_{1}\delta_{i}\left(G_{1}\right)\right)^{2}\left(n_{1}+n_{2}-1\right)\right]}}$
are eigenvalues of $\mathcal{L}\left(G_{1}\underset{=}{\vee}G_{2}\right)$
for each $i=2,3,...,n_{1}.$ 

So far we obtain $n_{2}-1+2\left(n_{1}-1\right)=2n_{1}+n_{2}-3$ eigenvalues
of $\mathcal{L}\left(G_{1}\underset{=}{\vee}G_{2}\right)$. The corresponding
eigenvectors are orthogonal to $\left(\mathbf{1}_{n_{i}}^{T},0_{1\times n_{1}},0_{1\times n_{2}}\right)^{T},$$\left(0_{1\times n_{1}},\mathbf{1}_{n_{i}}^{T},0_{1\times n_{2}}\right)^{T}$
and $\left(0_{1\times n_{1}},0_{1\times n_{1}},\mathbf{1}_{n_{2}}^{T}\right)^{T}.$
To find three additional eigenvalues, we look for eigenvectors of
$\mathcal{L}\left(G_{1}\underset{=}{\vee}G_{2}\right)$ of the form
$Y=\left(\alpha\mathbf{1}_{n_{1}}^{T},\beta\mathbf{1}_{n_{1}}^{T},\gamma\mathbf{1}_{n_{2}}^{T}\right)^{T}$
for $\left(\alpha,\beta,\gamma\right)\neq\left(0,0,0\right)$. Let
$x$ be an eigenvalue of $\mathcal{L}\left(G\underset{=}{\vee}G_{2}\right)$
corresponding to the eigenvector $Y$ . From $\mathcal{L}Y=xY$ we
get,
\begin{equation}
\alpha-\frac{\alpha r_{1}}{n_{1}+n_{2}-1}+\frac{\left(1+r_{1}-n_{1}\right)\beta}{\sqrt{n_{1}+n_{2}-1}\sqrt{n_{1}-r_{1}-1}}-\frac{n_{2}\gamma}{\sqrt{\left(n_{1}+n_{2}-1\right)\left(r_{2}+n_{1}\right)}}=\alpha x
\end{equation}
 
\begin{equation}
\frac{\left(1+r_{1}-n_{1}\right)\alpha}{\sqrt{n_{1}-r_{1}-1}\sqrt{n_{1}+n_{2}-1}}+\beta=\beta x
\end{equation}

\begin{equation}
\frac{-n_{1}\alpha}{\sqrt{r_{2}+n_{1}}\sqrt{n_{1}+n_{2}-1}}+\gamma-\frac{r_{2}\gamma}{r_{2}+n_{1}}\gamma x
\end{equation}
 Thus
\[
\alpha-\frac{\alpha r_{1}}{n_{1}+n_{2}-1}+\frac{\alpha(n_{1}-1-r_{1})}{\left(n_{1}+n_{2}-1\right)\left(x-1\right)}+\frac{\alpha(n_{1}n_{2})}{\left(n_{1}+n_{2}-1\right)\left(xr_{2}+n_{1}(x-1)\right)}=\alpha x.
\]
Notice that $\alpha\neq0$, since if $\alpha=0$ then $\alpha=\beta=\gamma=0$
and also $x\neq1$ since $x=1$ implies that $\alpha=0$.

Dividing by $\alpha$, we get 
\[
1-\frac{r_{1}}{n_{1}+n_{2}-1}+\frac{n_{1}-1-r_{1}}{\left(n_{1}+n_{2}-1\right)\left(x-1\right)}+\frac{n_{1}n_{2}}{\left(n_{1}+n_{2}-1\right)\left(xr_{2}+n_{1}(x-1)\right)}=x.
\]
Then, 
\begin{align*}
(n_{1}^{2}+n_{1}n_{2}-n_{1})(x-1)^{3}+(r_{2}n_{1}x+r_{2}n_{2}x-r_{2}x+r_{1}n_{1})(x-1)^{2}\\
+(r_{1}r_{2}x-n_{1}^{2}+n_{1}+r_{1}n_{1}-n_{1}n_{2})(x-1)-r_{2}(n_{1}-1-r_{1})x=0,
\end{align*}
so, by simple calculation we see that x is a root of the cubic equation
\begin{align*}
 & \left(n_{1}^{2}+n_{1}n_{2}-n_{1}+r_{2}n_{1}+r_{2}n_{2}-r_{2}\right)x^{3}-\left(3n_{1}^{2}+3n_{1}n_{2}-3n_{1}-r_{1}n_{1}+2r_{2}n_{1}+2r_{2}n_{2}-2r_{2}-r_{1}r_{2}\right)x^{2}\\
 & +\left(2n_{1}^{2}+2n_{1}n_{2}-2n_{1}-r_{1}n_{1}+r_{2}n_{2}\right)x=0
\end{align*}
and this completes the proof of (a).

\textbf{(b)} The proof of (i) is similar to the proof of (i) in (a).
Now we prove (ii), If $S=\left\{ 2,3,\ldots,n_{1}\right\} $, then
$\delta_{j}\left(G_{1}\right)=1+\frac{1}{r_{1}}$ for each $i=2,3,\ldots,n_{1}$,
so $\lambda_{i}\left(G_{1}\right)=-1$ for each $i=2,3,\ldots,n_{1}$,
i.e. $G_{1}=K_{n_{1}}$and $r_{1}=n_{1}-1$. So the normalized Laplacian
matrix of $G_{1}\underset{=}{\vee}G_{1}$ as follows:

\[
\mathcal{L}\left(G_{1}\underset{=}{\vee}G_{2}\right)=\left(\begin{array}{ccc}
I_{n_{1}}-\frac{A\left(G_{1}\right)}{n_{1}+n_{2}-1}\,\,\,\,\,\,\,\, & O_{n_{1}\times n_{1}}\,\,\,\,\,\,\,\, & \frac{-J_{n_{1}\times n_{2}}}{\sqrt{\left(n_{1}+n_{2}-1\right)\left(r_{2}+n_{1}\right)}}\\
\\
O_{n_{1}\times n_{1}} & O_{n_{1}\times n_{1}} & O_{n_{1}\times n_{2}}\\
\\
\frac{-J_{n_{2}\times n_{1}}}{\sqrt{\left(n_{1}+n_{2}-1\right)\left(r_{2}+n_{1}\right)}} & O_{n_{2}\times n_{1}} & I_{n_{2}}-\frac{A\left(G_{2}\right)}{r_{2}+n_{1}}
\end{array}\right).
\]

For $i=2,...,n_{1}$, let $X_{i}$ be an eigenvector of $A(G_{1})$
corresponding to the eigenvalue $\lambda_{i}\left(G_{1}\right)=-1$.
So $\left(\begin{array}{ccc}
X_{i}^{T}\,\,\,\, & 0_{1\times n_{1}}\,\,\, & 0_{1\times n_{2}}\end{array}\right)^{T}$ is an eigenvector of $\mathcal{L}(G_{1}\underset{=}{\vee}G_{2})$
corresponding to the eigenvalue $1+\frac{1}{n_{1}+n_{2}-1}$ and $\left(\begin{array}{ccc}
0_{1\times n_{1}}\,\,\,\, & X_{i}^{T}\,\,\, & 0_{1\times n_{2}}\end{array}\right)^{T}$ is an eigenvector of $\mathcal{L}(G_{1}\underset{=}{\vee}G_{2})$
corresponding to the eigenvalue $0$ because,\\

$\mathcal{L}\left(\begin{array}{c}
X_{i}\\
0_{n_{1}\times1}\\
0_{n_{2}\times1}
\end{array}\right)=\left(\begin{array}{c}
X_{i}-\frac{\lambda_{i}\left(G_{1}\right)X_{i}}{n_{1}+n_{2}-1}\\
0_{n_{1}\times1}\\
0_{n_{2}\times1}
\end{array}\right)=\left(1+\frac{1}{n_{1}+n_{2}-1}\right)\left(\begin{array}{c}
X_{i}\\
0_{n_{1}\times1}\\
0_{n_{2}\times1}
\end{array}\right)$ and $\mathcal{L}\left(\begin{array}{c}
0_{n_{1}\times1}\\
X_{i}\\
0_{n_{2}\times1}
\end{array}\right)=\left(\begin{array}{c}
0_{n_{1}\times1}\\
0_{n_{1}\times1}\\
0_{n_{2}\times1}
\end{array}\right)$. \\

Therefore, $(1+\frac{1}{n_{1}+n_{2}-1})^{[n_{1}-1]}$ and $0^{[n_{1}-1]}$
are eigenvalues of $\mathcal{L}(G_{1}\underset{=}{\vee}G_{2})$.

So far, we obtained $n_{2}-1+2\left(n_{1}-1\right)=2n_{1}+n_{2}-3$
eigenvalues of $\mathcal{L}\left(G_{1}\underset{=}{\vee}G_{2}\right)$.
Their eigenvectors are orthogonal to $\left(\mathbf{1}_{n_{1}}^{T},0_{1\times n_{1}},0_{1\times n_{2}}\right)^{T},$$\left(0_{1\times n_{1}},1_{n_{1}}^{T},0_{1\times n_{2}}\right)^{T}$
and $\left(0_{1\times n_{1}},0_{1\times n_{1}},1_{n_{2}}^{T}\right)^{T}.$
To find three additional eigenvalues, we look for eigenvectors of
$\mathcal{L}\left(G\underset{=}{\vee}G_{2}\right)$ of the form $Y=\left(\alpha\mathbf{1}_{n_{1}}^{T},\beta\mathbf{1}_{n_{1}}^{T},\gamma\mathbf{1}_{n_{2}}^{T}\right)^{T}$for
$\left(\alpha,\beta,\gamma\right)\neq\left(0,0,0\right)$. Let $x$
be an eigenvalue of $\mathcal{L}\left(G\underset{=}{\vee}G_{2}\right)$
corresponding to the eigenvector $Y$ . Then from $\mathcal{L}Y=xY$
we get 

\begin{equation}
\alpha-\frac{r_{1}\alpha}{n_{1}+n_{2}-1}-\frac{\gamma n_{2}}{\sqrt{\left(n_{1}+n_{2}-1\right)\left(r_{2}+n_{1}\right)}}=\alpha x
\end{equation}

\begin{equation}
\beta x=0
\end{equation}

\begin{equation}
\frac{-\alpha n_{1}}{\sqrt{\left(n_{1}+n_{2}-1\right)\left(r_{2}+n_{1}\right)}}+\gamma-\frac{r_{2}\gamma}{r_{2}+n_{1}}=\gamma x
\end{equation}

If $\beta\neq0$, then $\left(0,\beta,0\right)$ is a one of the solutions
of the three above equations, so $\left(0_{1\times n_{1}},\beta\mathbf{1}_{n_{1}}^{T},0_{1\times n_{2}}\right)^{T}$
is an eigenvector corresponding to the eigenvalue $0$. On the other
hand if $\beta=0$ we get 
\begin{equation}
\alpha\left(1-x-\frac{r_{1}}{n_{1}+n_{2}-1}\right)=\frac{\gamma n_{2}}{\sqrt{\left(n_{1}+n_{2}-1\right)\left(r_{2}+n_{1}\right)}}
\end{equation}

\begin{equation}
\gamma\left(1-x-\frac{r_{2}}{r_{2}+n_{1}}\right)=\frac{\alpha n_{1}}{\sqrt{\left(n_{1}+n_{2}-1\right)\left(r_{2}+n_{1}\right)}}
\end{equation}

By solving above two equations we get the equation 
\[
\left(r_{2}+n_{1}\right)\left(n_{1}+n_{2}-1\right)x^{2}+\left(n_{1}-r_{2}n_{2}-2n_{1}n_{2}-n_{1}^{2}\right)x=0,
\]
 whose roots are $0$ and $\frac{n_{1}^{2}+2n_{1}n_{2}+r_{2}n_{2}-n_{1}}{(r_{2}+n_{1})(n_{1}+n_{2}-1)}$.

This completes the proof of (b).

\textbf{(c)} The proofs of (i), (ii) and (iv) are similar to the proofs
of (i), (ii) and (iii) of (a), respectively. Now we prove (iii), Let
$S$$\neq\varPhi$ and $S\neq\left\{ 2,3,\ldots,n_{1}\right\} $.
If $i\in S$, then $1+\frac{1}{n_{1}+n_{2}-1}$ and $1$ are eigenvalues
of $\mathcal{L}\left(G\underset{=}{\vee}G_{2}\right)$ because if
$X_{i}$ is an eigenvector corresponding to the eigenvalue $\delta_{i}(G_{1})$,
then\\

$\mathcal{L}\left(\begin{array}{c}
X_{i}\\
0_{n_{1}\times1}\\
0_{n_{2}\times1}
\end{array}\right)=\left(\begin{array}{c}
X_{i}+\frac{X_{i}}{n_{1}+n_{2}-1}\\
0_{n_{1}\times1}\\
0_{n_{2}\times1}
\end{array}\right)=\left(1+\frac{1}{n_{1}+n_{2}-1}\right)\left(\begin{array}{c}
X_{i}\\
0_{n_{1}\times1}\\
0_{n_{2}\times1}
\end{array}\right)$ and $\mathcal{L}\left(\begin{array}{c}
0_{n_{1}\times1}\\
X_{i}\\
0_{n_{2}\times1}
\end{array}\right)=\left(\begin{array}{c}
0_{n_{1}\times1}\\
X_{i}\\
0_{n_{2}\times1}
\end{array}\right)$.\\

So, $1^{[n(S)]}$ and $(1+\frac{1}{n_{1}+n_{2}-1})^{[n(S)]}$ are
eigenvalues of $\mathcal{L}\left(G\underset{=}{\vee}G_{2}\right)$
and this completes the proof of (c). 
\end{proof}
Now we can give another answer to Question \textbf{\ref{que:Construct-non-regular}}
by constructing several pairs of non regular $\left\{ A,L,Q,\mathcal{L}\right\} $NICS
graphs.
\begin{cor}
\textup{\label{thm:result}}\textup{\emph{Let $G_{1}$, $H_{1}$ be
cospectral regular graphs and $G_{2}$, $H_{2}$ be non isomorphic,
regular and cospectral graphs. Then $G_{1}\underset{=}{\lor}G_{2}$
and $H_{1}\underset{=}{\lor}H_{2}$ are non regular $\left\{ A,L,Q,\mathcal{L}\right\} $NICS.}}
\end{cor}

\begin{proof}
\emph{$G_{1}\underset{=}{\lor}G_{2}$ }and\emph{ $H_{1}\underset{=}{\lor}H_{2}$}
are non isomorphic since $G_{2}$ and $H_{2}$ are non isomorphic.
By Theorems \textbf{\ref{thm:A1}}, \textbf{\ref{thm:L1}}, \textbf{\ref{thm:Q1}}
and \textbf{\ref{thm: normalized Laplacian  of NNS for regular graph}}
, we get that $G_{1}\underset{=}{\lor}G_{2}$ and $H_{1}\underset{=}{\lor}H_{2}$
are non regular $\left\{ A,L,Q,\mathcal{L}\right\} $NICS. 
\end{proof}
\begin{example}
Let $G_{1}=H_{1}=C_{4}$, and choose $G_{2}=G$ and $H_{2}=H$ where
$G$ and $H$ are graphs in Figure \textbf{\ref{fig:Two-regular-non isomorphic cospectral  10 vertices}},
then the graphs in Figure \textbf{\ref{fig:Non-regular of NNS (A,L ,Q ,l)}}
are $\left\{ A,L,Q,\mathcal{L}\right\} $NICS. 

\begin{figure}[h]
\begin{centering}
\subfloat[$C_{4}\protect\underset{=}{\vee}G_{2}$]{\includegraphics[width=0.5\textwidth]{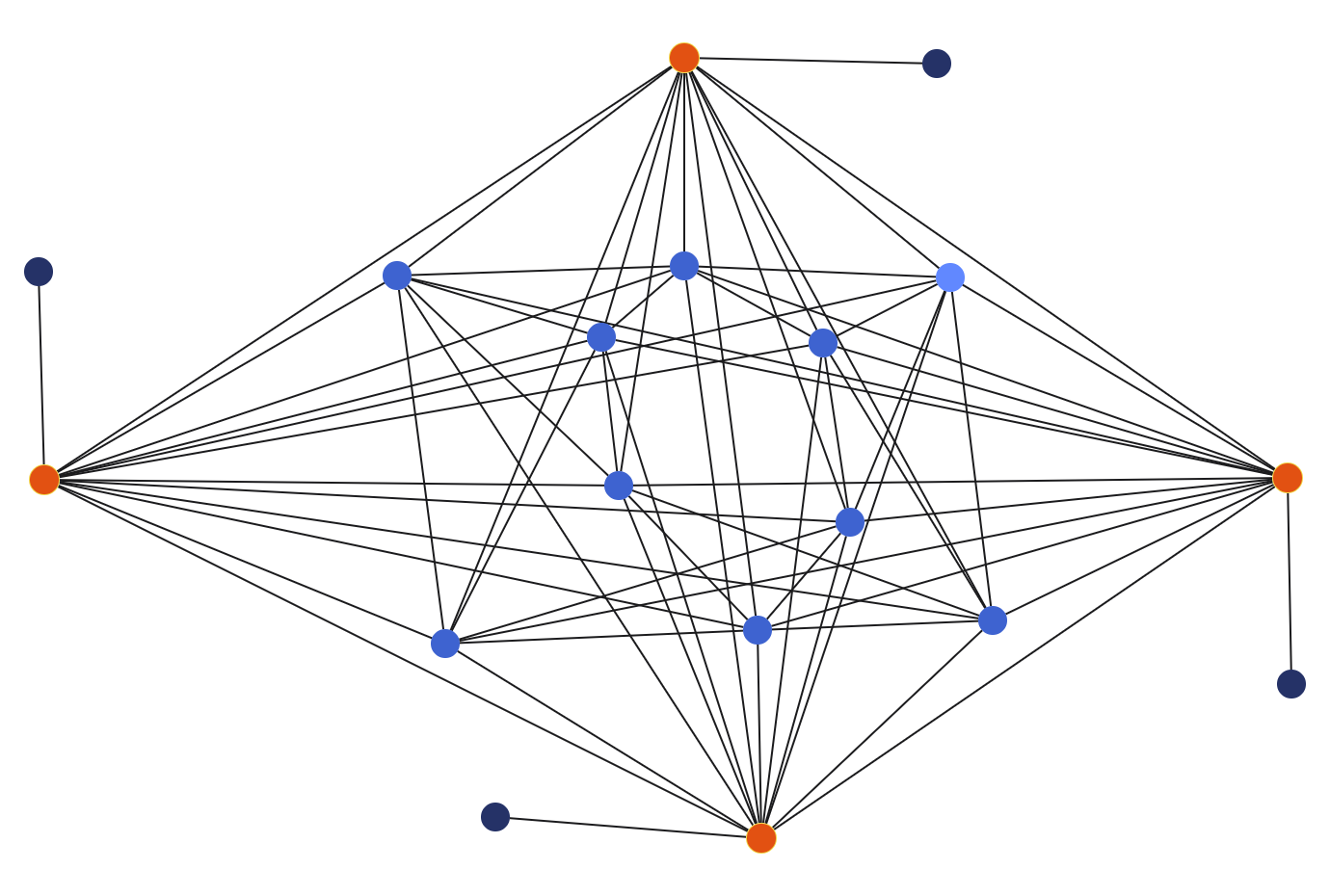}

}\subfloat[$C_{4}\protect\underset{=}{\vee}H_{2}$]{\includegraphics[width=0.5\textwidth]{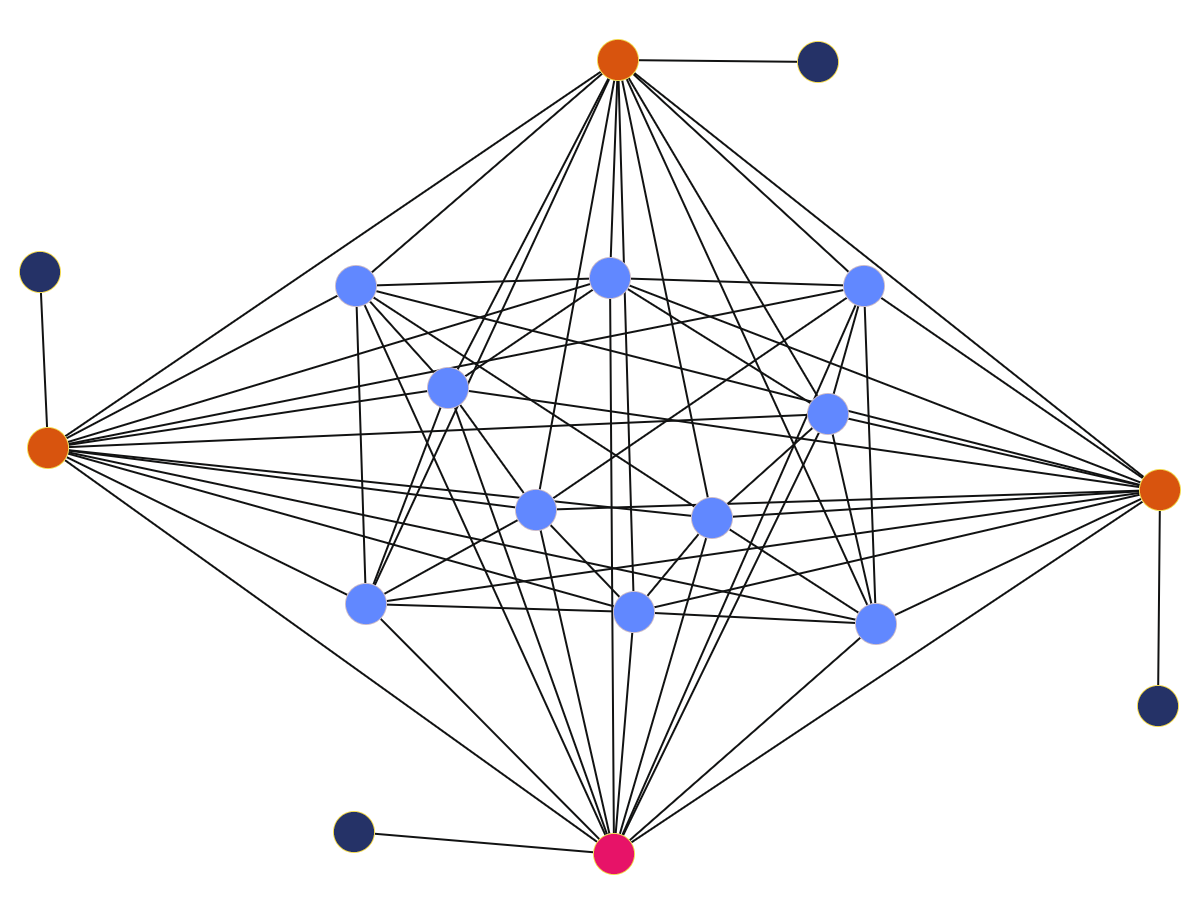}
}

\par\end{centering}
\caption{\label{fig:Non-regular of NNS (A,L ,Q ,l)}Non regular $\left\{ A,L,Q,\mathcal{L}\right\} $NICS
graphs. }

\newpage{}
\end{figure}
\end{example}

\bibliographystyle{plain}
\bibliography{Ref}

\begin{thebibliography}{10}

\bibitem{barik2015laplacian}
Sasmita Barik, Ravindra~B Bapat, and Sukanta Pati.
\newblock On the laplacian spectra of product graphs.
\newblock {\em Applicable Analysis and Discrete Mathematics}, pages 39--58,
  2015.

\bibitem{barik2007spectrum}
Sasmita Barik, Sukanta Pati, and BK~Sarma.
\newblock The spectrum of the corona of two graphs.
\newblock {\em SIAM Journal on Discrete Mathematics}, 21(1):47--56, 2007.

\bibitem{brouwer2011spectra}
Andries~E Brouwer and Willem~H Haemers.
\newblock {\em Spectra of graphs}.
\newblock Springer Science \& Business Media, 2011.

\bibitem{butler2010note}
Steve Butler.
\newblock A note about cospectral graphs for the adjacency and normalized
  laplacian matrices.
\newblock {\em Linear and Multilinear Algebra}, 58(3):387--390, 2010.

\bibitem{cardoso2013spectra}
Domingos~M Cardoso, Maria Aguieiras~A de~Freitas, Enide~Andrade Martins, and
  Mar{\'\i}a Robbiano.
\newblock Spectra of graphs obtained by a generalization of the join graph
  operation.
\newblock {\em Discrete Mathematics}, 313(5):733--741, 2013.

\bibitem{chung1997spectral}
Fan~RK Chung and Fan~Chung Graham.
\newblock {\em Spectral graph theory}.
\newblock Number~92. American Mathematical Soc., 1997.

\bibitem{cui2012spectrum}
Shu-Yu Cui and Gui-Xian Tian.
\newblock The spectrum and the signless laplacian spectrum of coronae.
\newblock {\em Linear algebra and its applications}, 437(7):1692--1703, 2012.

\bibitem{cvetkovivc2010introduction}
D~Cvetkovi{\v{c}}, P~Rowlinson, and S~Simi{\'c}.
\newblock An introduction to the theory of graph spectra london mathematical
  society student texts.
\newblock {\em London: Cambridge University}, 2010.

\bibitem{cvetkovic1975spectra}
Drago{\v{s}} Cvetkovic.
\newblock Spectra of graphs formed by some unary operations.
\newblock {\em Publ. Inst. Math.(Beograd)}, 19(33):37--41, 1975.

\bibitem{cvetkovic1971graphs}
Drago{\v{s}}~M Cvetkovi{\'c}.
\newblock Graphs and their spectra.
\newblock {\em Publikacije Elektrotehni{\v{c}}kog fakulteta. Serija Matematika
  i fizika}, (354/356):1--50, 1971.

\bibitem{das2019new}
A~Das and P~Panigrahi.
\newblock New classes of simultaneous cospectral graphs for adjacency,
  laplacian and normalized laplacian matrices.
\newblock {\em Kragujevac Journal of Mathematics}, 43(2):303--323, 2019.

\bibitem{godsil1978new}
CD~Godsil and BD~McKay.
\newblock A new graph product and its spectrum.
\newblock {\em Bulletin of the Australian Mathematical Society}, 18(1):21--28,
  1978.

\bibitem{gopalapillai2011spectrum}
Indulal Gopalapillai.
\newblock The spectrum of neighborhood corona of graphs.
\newblock {\em Kragujevac journal of mathematics}, 35(3):493--500, 2011.

\bibitem{harary1969graph}
F~Harary.
\newblock Graph theory. addison wesley publishing company.
\newblock {\em Reading, Massachusetts}, 1969.

\bibitem{horn2005basic}
Roger~A Horn and Fuzhen Zhang.
\newblock Basic properties of the schur complement.
\newblock {\em The Schur Complement and Its Applications}, pages 17--46, 2005.

\bibitem{hou2010spectrum}
Yaoping Hou and Wai-Chee Shiu.
\newblock The spectrum of the edge corona of two graphs.
\newblock {\em The Electronic Journal of Linear Algebra}, 20:586--594, 2010.

\bibitem{indulal2012spectrum}
Gopalapillai Indulal.
\newblock Spectrum of two new joins of graphs and infinite families of integral
  graphs.
\newblock {\em Kragujevac journal of mathematics}, 36(38):133--139, 2012.

\bibitem{liu2019spectra}
Xiaogang Liu and Zuhe Zhang.
\newblock Spectra of subdivision-vertex join and subdivision-edge join of two
  graphs.
\newblock {\em Bulletin of the Malaysian Mathematical Sciences Society},
  42(1):15--31, 2019.

\bibitem{lu2023spectra}
Zhiqin Lu, Xiaoling Ma, and Minshao Zhang.
\newblock Spectra of graph operations based on splitting graph.
\newblock {\em Journal of Applied Analysis \& Computation}, 13(1):133--155,
  2023.

\bibitem{mcleman2011spectra}
Cam McLeman and Erin McNicholas.
\newblock Spectra of coronae.
\newblock {\em Linear algebra and its applications}, 435(5):998--1007, 2011.

\bibitem{mohar1992laplace}
Bojan Mohar.
\newblock Laplace eigenvalues of graphs—a survey.
\newblock {\em Discrete mathematics}, 109(1-3):171--183, 1992.

\bibitem{nica2016brief}
Bogdan Nica.
\newblock A brief introduction to spectral graph theory.
\newblock {\em arXiv preprint arXiv:1609.08072}, 2016.

\bibitem{pavithra2021spectra}
R~Pavithra and R~Rajkumar.
\newblock Spectra of m-edge rooted product of graphs.
\newblock {\em Indian Journal of Pure and Applied Mathematics},
  52(4):1235--1255, 2021.

\bibitem{pavithra2022spectra}
R~Pavithra and R~Rajkumar.
\newblock Spectra of bowtie product of graphs.
\newblock {\em Discrete Mathematics, Algorithms and Applications},
  14(02):2150114, 2022.

\bibitem{rajkumar2019spectra}
R~Rajkumar and M~Gayathri.
\newblock Spectra of (h1, h2)-merged subdivision graph of a graph.
\newblock {\em Indagationes Mathematicae}, 30(6):1061--1076, 2019.

\bibitem{rajkumar2022spectra}
R~Rajkumar and R~Pavithra.
\newblock Spectra of m-rooted product of graphs.
\newblock {\em Linear and Multilinear Algebra}, 70(1):1--26, 2022.

\bibitem{spielman2007spectral}
Daniel~A Spielman.
\newblock Spectral graph theory and its applications.
\newblock In {\em 48th Annual IEEE Symposium on Foundations of Computer Science
  (FOCS'07)}, pages 29--38. IEEE, 2007.

\bibitem{van2003graphs}
Edwin~R Van~Dam and Willem~H Haemers.
\newblock Which graphs are determined by their spectrum?
\newblock {\em Linear Algebra and its applications}, 373:241--272, 2003.

\bibitem{van2009developments}
Edwin~R Van~Dam and Willem~H Haemers.
\newblock Developments on spectral characterizations of graphs.
\newblock {\em Discrete Mathematics}, 309(3):576--586, 2009.

\bibitem{varghese2020normalized}
Renny~P Varghese and D~Susha.
\newblock On the normalized laplacian spectrum of some graphs.
\newblock {\em Kragujevac Journal of Mathematics}, 44(3):431--442, 2020.

\end{thebibliography}

\end{document}